\documentclass[12pt]{amsart}

\usepackage{amsmath,amssymb,fullpage,color,graphicx}
\usepackage{hyperref}
\hypersetup{colorlinks=true,urlcolor=blue,citecolor=blue,linkcolor=blue}
\usepackage{colonequals}
\usepackage{tikz}   
\usepackage{tikz-cd}
\usepackage{longtable}
\usepackage{supertabular}
\usepackage{tabularx}
\usepackage{multicol}
\usepackage{multirow}
\usepackage{enumitem}
\usepackage[utf8]{inputenc}

\numberwithin{equation}{subsection}

\theoremstyle{plain}
\newtheorem{lemma}[equation]{Lemma}
\newtheorem{prop}[equation]{Proposition}

\newtheorem{corollary}[equation]{Corollary}
\newtheorem{theorem}[equation]{Theorem}

\newtheorem{thm}{Theorem}[section] 
\theoremstyle{remark}

\newtheorem{remark}[equation]{Remark}

\theoremstyle{definition}
\newtheorem{definition}[equation]{Definition}

\newenvironment{enumalph}
{\begin{enumerate}}
{\end{enumerate}}

\newenvironment{enumroman}
{\begin{enumerate}}
{\end{enumerate}}

\newcommand{\C}{\mathbb C}
\newcommand{\F}{\mathbb F}

\newcommand{\Q}{\mathbb Q}
\newcommand{\R}{\mathbb R}

\newcommand{\Z}{\mathbb Z}

\newcommand{\Qbar}{\overline{\mathbb{Q}}}

\newcommand{\Fbar}{\bar{F}}

\newcommand{\calO}{\mathcal O}
\newcommand{\calA}{\mathcal A}

\newcommand{\quat}[2]{\displaystyle{\biggl(\frac{#1}{#2}\biggr)}}

\DeclareMathOperator{\NS}{NS}
\DeclareMathOperator{\Mat}{Mat}

\DeclareMathOperator{\disc}{disc}
\DeclareMathOperator{\End}{End}
\DeclareMathOperator{\Gal}{Gal}
\DeclareMathOperator{\GL}{GL}

\DeclareMathOperator{\nrd}{nrd}
\DeclareMathOperator{\trd}{trd}

\DeclareMathOperator{\Pic}{Pic}

\DeclareMathOperator{\rk}{rk}

\DeclareMathOperator{\tors}{tors}
\DeclareMathOperator{\red}{red}
\DeclareMathOperator{\Aut}{Aut}

\DeclareMathOperator{\Jac}{Jac}
\DeclareMathOperator{\Hom}{Hom}

\newcommand{\lmfdbmf}[1]{\href{http://www.lmfdb.org/ModularForm/GL2/Q/holomorphic/#1}{\textsf{#1}}}

\newcommand{\defi}{\textsf}

\begin{document}

\title{Rational torsion points on abelian surfaces with quaternionic multiplication}

\author{Jef Laga}
\address{Department of Mathematics, Princeton University, Princeton, NJ 08544, USA}
\email{jeflaga@hotmail.com}

\author{Ciaran Schembri}
\address{Department of Mathematics, Dartmouth College, 6188 Kemeny Hall, Hanover, NH 03755, USA}
\email{ciaran.schembri@dartmouth.edu}

\author{Ari Shnidman}
\address{Einstein Institute of Mathematics, Hebrew University of Jerusalem, Israel} 
\email{ari.shnidman@gmail.com}

\author{John Voight}
\address{Department of Mathematics, Dartmouth College, 6188 Kemeny Hall, Hanover, NH 03755, USA}
\email{jvoight@gmail.com}

\begin{abstract}
Let $A$ be an abelian surface over $\Q$ whose geometric endomorphism ring is a maximal order in a non-split quaternion algebra. Inspired by Mazur's theorem for elliptic curves, we show that the torsion subgroup of $A(\Q)$ is $12$-torsion and has order at most $18$. Under the additional assumption that $A$ is of $\GL_2$-type, we give a complete classification of the possible torsion subgroups of $A(\Q)$.    
\end{abstract}

\maketitle
\bibliographystyle{alpha}

\tableofcontents

\section{Introduction}

\subsection{Motivation}

Let $E$ be an elliptic curve over $\Q$.
In \cite{Mazur77}, Mazur famously showed that if a prime $\ell$ divides the order of the torsion subgroup $E(\Q)_{\tors}$ of $E(\Q)$ then $\ell \leq 7$.
Combining with previous work of Kubert \cite{Kubert76}, Mazur deduced that $\#E(\Q)_{\tors} \leq 16$ and that $E(\Q)_{\mathrm{tors}}$ is isomorphic to one of fifteen finite abelian groups, each of which gives rise to a genus $0$ modular curve with a well known rational parametrization.

It is not known whether there is a uniform bound on the size of the rational torsion subgroup of abelian varieties of fixed dimension $g\geq 2$ over a fixed number field. 
In fact, there is not even a single integer $N$ for which it is known that there is no abelian surface over $\Q$ with a torsion point of order $N$. Indeed, determining rational points on Siegel modular threefolds with level structure seems challenging in general.

\subsection{Results}
In this paper we study the torsion subgroup of abelian surfaces $A$ over $\Q$ whose geometric endomorphism ring is large. 
Namely, we suppose that the geometric endomorphism ring $\End(A_{\overline{\Q}})$ is a maximal order $\calO$ in a division quaternion algebra over $\Q$; we refer to these as $\calO$-$\mathrm{PQM}$ surfaces (``potential quaternionic multiplication'').  
Such abelian surfaces are geometrically simple, so their torsion subgroup cannot be studied using torsion subgroups of elliptic curves. 
On the other hand, they give rise to rational points on certain Shimura curves, much as elliptic curves over $\Q$ give rise to rational points on modular curves. 
Thus $\calO$-$\mathrm{PQM}$ surfaces are a natural place to explore torsion questions in higher dimension.

Our main results show that the torsion behaviour of $\calO$-PQM surfaces is rather constrained.

\begin{thm}\label{thm:QMmazur ptorsion}
Let $A$ be an $\calO$-$\mathrm{PQM}$ abelian surface over $\Q$ with a rational point of order $\ell$, where $\ell$ is a prime number. Then $\ell=2$ or $\ell=3$.
\end{thm}

\begin{thm}\label{thm:optimalbound}
Each $\calO$-$\mathrm{PQM}$ abelian surface $A$ over $\Q$ has $\#A(\Q)_{\tors} \leq 18$.
\end{thm}

The fact that the rational torsion on $\calO$-PQM surfaces is uniformly bounded is not new nor is it difficult to prove. Indeed, since $\calO$-$\mathrm{PQM}$ surfaces have everywhere potentially good reduction (Lemma \ref{lemma:PQMreduction}), local methods quickly show that  $\ell \mid \#A(\Q)_{\tors}$ implies $\ell \leq 19$ 
and that $\#A(\Q)_{\tors} \leq 72$ \cite[Theorem 1.4]{CX08}.  The goal of this paper is instead to prove results which are as precise as possible.

Theorems \ref{thm:QMmazur ptorsion} and \ref{thm:optimalbound} are optimal since it is known that each of the seven groups 
\begin{equation}\label{eq:LS examples}
\begin{gathered}
\{1\}, \,  \Z/2\Z, \, \Z/3\Z, \, (\Z/2\Z)^2\\
\Z/6\Z, \, (\Z/3\Z)^2,\, \Z/2\Z \times (\Z/3\Z)^2
\end{gathered}
\end{equation}
is isomorphic to $A(\Q)_{\tors}$ for some $\calO$-$\mathrm{PQM}$ surface $A/\Q$, with the largest group having order $18$. Indeed, each of these groups arises as $A(\Q)_{\mathrm{tors}}$ for infinitely many $\Qbar$-isomorphism classes of such surfaces by \cite[Theorem 1.1]{LagaShnidman}. 

Our methods give the following more precise constraints on the group structure of $A(\Q)_{\tors}$.

\begin{thm}\label{thm:QMmazur_groups}
Let $A$ be an $\calO$-$\mathrm{PQM}$ abelian surface over $\Q$. Then $A(\Q)_{\tors}$ is isomorphic either to one of the groups in \eqref{eq:LS examples} or to one of the following groups: 
\begin{equation}\label{eq:unknowns}
\begin{gathered}\Z/4\Z, \Z/2\Z \times \Z/4\Z, (\Z/2\Z)^3, (\Z/2\Z)^2 \times \Z/3\Z,\\ 
\Z/4\Z \times \Z/3\Z, (\Z/2\Z)^2 \times \Z/4\Z, (\Z/4\Z)^2.
  \end{gathered}
  \end{equation}
\end{thm}
We leave open the question of whether any of the groups of \eqref{eq:unknowns} arise as $A(\Q)_{\mathrm{tors}}$ for some $\calO$-$\mathrm{PQM}$ surface or not.

Theorem \ref{thm:QMmazur_groups} can be interpreted as a non-existence result for non-special  rational points on certain types of Shimura curves with level structure. Since the discriminant of $\End(A_{\Qbar})$ and level are unconstrained, the result covers infinitely many distinct such curves.   
However, as we explain below, our proof of Theorem \ref{thm:QMmazur_groups} does not make direct use of the arithmetic of Shimura curves.

Whereas the theorems above consider general $\calO$-$\mathrm{PQM}$ abelian surfaces, one is sometimes interested in surfaces with additional structure.  For example, recall that $A$ is of \defi{$\GL_2$-type} if the endomorphism ring $\End(A)$ is a quadratic ring.  Modularity results (see Theorem \ref{theorem: modularity theorem}) imply that an abelian variety $A$ of $\GL_2$-type over $\Q$ is a quotient of the modular Jacobian $J_1(N)$ for some $N$. More precisely, the isogeny class of $A$ arises from a cuspidal newform of weight $2$ and level $N$, where $A$ has conductor $N^2$. 
Specializing our methods to this setting, we obtain the following complete classification.

\begin{thm}\label{thm:gl2type classification}
 Let $A$ be an $\calO$-$\mathrm{PQM}$ surface over $\Q$ of $\GL_2$-type. Then $A(\Q)_{\tors}$ is isomorphic to one of  the following groups:
 \[\{1\}, \Z/2\Z, \Z/3\Z, (\Z/2\Z)^2, (\Z/3\Z)^2. \]
Every one of these groups arises as $A(\Q)_{\tors}$ for infinitely many $\overline{\Q}$-isomorphism classes of $\calO$-$\mathrm{PQM}$ surfaces $A$ over $\Q$ of $\GL_2$-type. 
\end{thm}

\begin{remark}The proof shows that if the maximality assumption on $\calO$ is omitted, then a similar classification holds except we do not know whether the group $(\Z/2\Z)^3$ arises or not. 
\end{remark}

Another natural class of abelian surfaces is Jacobians of genus two curves. Recall that for geometrically simple abelian surfaces, being a Jacobian is equivalent to carrying a principal polarization. Thus, the following result gives a near-classification for rational torsion subgroups of genus two Jacobians over $\Q$ in the $\calO$-$\mathrm{PQM}$ locus of the Siegel modular 3-fold  $\mathcal{A}_2$ parameterizing principally polarized abelian surfaces.

\begin{thm}\label{thm:PQM Jacobians}
 Let $J$ be an $\calO$-$\mathrm{PQM}$ surface over $\Q$ which is the Jacobian of a genus two curve over $\Q$. Then $J(\Q)_{\tors}$ is isomorphic to one of the following groups:
  \begin{gather*} \{1\}, \Z/2\Z, \Z/3\Z, (\Z/2\Z)^2,\Z/6\Z, (\Z/3\Z)^2, \\ \Z/4\Z, \Z/2\Z \times \Z/4\Z, (\Z/2\Z)^2 \times \Z/3\Z, \Z/4\Z \times \Z/3\Z, (\Z/4\Z)^2 
  \end{gather*}
   In particular, $\# J(\Q)_{\tors} \leq 16$.
\end{thm}

The first six groups in the list above can be realized as $J(\Q)_{\tors}$; see Table \ref{PQMequationstable}. We do not know whether they can be realized infinitely often by $\calO$-$\mathrm{PQM}$ Jacobians over $\Q$.

\subsection{Methods}\label{subsec: methods}

We first describe the proof of Theorem \ref{thm:gl2type classification}, which is almost entirely local in nature.  Let $A$ be an $\calO$-$\mathrm{PQM}$ surface over $\Q$ of $\GL_2$-type. 
We show that $A$ has totally additive reduction at every prime $p$ of bad reduction, meaning that the identity component of the special fiber of the N\'eron model at $p$ is unipotent.
It is well known that in this case the prime-to-$p$ torsion subgroup of $A(\Q_p)$ embeds in the N\'eron component group of $A$ at $p$, and that this component group is controlled by the smallest field over which $A$ acquires good reduction. 
Our proof of Theorem \ref{thm:gl2type classification} therefore involves an analysis of this field extension, in particular we show that its degree is coprime to $\ell$ for any prime $\ell\geq 5$.
Applying these local arguments requires the existence of suitable primes of bad reduction, and breaks down when $A$ has conductor of the form $2^n$, $3^n$, or $6^4$. We handle these cases seperately by invoking the modularity theorem. 
It turns out there is a single isogeny class whose conductor is of this form, namely the isogeny class of conductor $3^{10}$, corresponding to a Galois orbit of newforms of level $3^5=243$, with LMFDB label \lmfdbmf{243.2.a.d}.  

To prove Theorem \ref{thm:QMmazur ptorsion}, we need to exclude the existence of an $\calO$-$\mathrm{PQM}$ surface $A$ over $\Q$ such that $A[\ell](\Q)$ is nontrivial for some prime $\ell \geq 5$.
By studying the interaction between the Galois action on the torsion points of $A$ and the Galois action on $\End(A_{\Qbar})$, we show that such an $A$ must necessarily be of $\GL_2$-type, so we may conclude using Theorem \ref{thm:gl2type classification}. The methods of this `reduction to $\GL_2$-type' argument are surprisingly elementary. Aside from some calculations in the quaternion order $\calO$, the key observation is that in the non-$\GL_2$-type case, the geometric endomorphism algebra $\End^0(A_{\Qbar})$ contains a (unique) Galois-stable imaginary quadratic subfield, which is naturally determined by the (unique) polarization defined over $\Q$.

To prove Theorems \ref{thm:optimalbound} and \ref{thm:QMmazur_groups}, we must constrain the remaining possibilities for $A(\Q)_{\tors}$, which is a group of order $2^i3^j$ by Theorem \ref{thm:QMmazur ptorsion}. 
Our arguments here are ad hoc, relying on a careful analysis of the reduction of $A$ modulo various primes via Honda--Tate theory (with the aid of the LMFDB \cite{lmfdb}) to constrain the possible torsion groups, reduction types, and Galois action on the endomorphism ring. 
The proof of Theorem \ref{thm:PQM Jacobians} is similar, but using the relationship between endomorphisms, polarizations, and level structures.



\subsection{Previous work}
Rational torsion on $\calO$-$\mathrm{PQM}$ surfaces was  previously considered in the Ph.D.\ thesis of Clark \cite[Chapter 5]{ClarkThesis}, but see the author's \href{http://alpha.math.uga.edu/~pete/papers.html}{caveat emptor}, indicating that the proofs of the main results of that chapter are incomplete.

\subsection{Future directions}
 Our methods use the maximality assumption on $\End(A_{\Qbar})$ in various places.  It would be interesting and desirable to relax this condition, especially since groups of order 12 and 18 can indeed arise in genus two Jacobians with non-maximal PQM; see, for example, \href{https://www.lmfdb.org/Genus2Curve/Q/20736/l/373248/1}{the curve} $y^2 = 24x^5 + 36x^4- 4x^3-12x^2+1$ and \cite[Remark 7.17]{LagaShnidman}. 
It would also be interesting to systematically analyze rational points on (Atkin--Lehner quotients of) Shimura curves with level structure, for example to determine whether the remaining groups \eqref{eq:unknowns} arise or not. We hope to address this in future work.

\subsection{Organization}
Sections \ref{section: quaternionic algebra}-\ref{section:PQM surfaces over localfinite fields} are preliminary, and the remaining sections are devoted to proving the main theorems of the introduction.
As explained in \S\ref{subsec: methods}, we start by proving Theorem \ref{thm:gl2type classification} because the other theorems depend on it.

Those who wish to take the shortest route to Theorem \ref{thm:gl2type classification} (minus eliminating $(\Z/2\Z)^3$) only need to read Sections \ref{subsection: endomorphism field}, \ref{section:PQM surfaces over localfinite fields} and \ref{section: torsion gl2type}.
Eliminating the last group $(\Z/2\Z)^3$ in Proposition \ref{prop: GL2 type no (Z/2)^3 torsion} requires more algebraic preliminaries from Section \ref{section: quaternionic algebra} and \ref{section: galois actions, polarizations and endomorphisms}.

\subsection{Acknowledgements}

We would like to thank Davide Lombardo for interesting discussions related to this project.
Schembri and Voight were supported by a Simons Collaboration Grant (550029, to JV). 
Part of this project was carried out while Laga visited Shnidman at the Hebrew University of Jerusalem. Shnidman was funded by the European Research Council (ERC, CurveArithmetic, 101078157).

\subsection{Notation}
We fix the following notation for the remainder of this paper:
\begin{itemize}
    \item $B$: an indefinite (so $B\otimes \mathbb{R} \simeq \Mat_2(\mathbb{R})$) quaternion algebra over $\Q$ of discriminant $\disc(B)\neq 1$;
    \item $\trd(b)$, $\nrd(b)$ and $\bar{b}$: reduced trace, reduced norm, and canonical involution of an element $b\in B$ respectively; 
    \item $\calO$: a choice of maximal order of $B$; 
    \item $\bar F$: a choice of algebraic closure of a field $F$;
    \item $\Gal_F$: the absolute Galois group of $F$;
    \item $\End(A)$: the endomorphism ring of an abelian variety $A$ defined over $F$;
    \item $\End^0(A)=\End(A) \otimes \Q$: the endomorphism algebra of $A$;
    \item $\NS(A)$: the N\'eron--Severi group of $A$;
    \item $A_{K}$: base change of $A/F$ along a field extension $K/F$;
    \item $\left(\frac{m,n}{F}\right)$: the quaternion algebra over $F$ with basis $\{1,i,j,ij\}$ such that $i^2= m, j^2=n$ and $ij = -ji$;
    \item $D_n:$ the dihedral group of order $2n$.
\end{itemize}
We say an abelian surface $A$ over a field $F$ is an \defi{$\calO$-$\mathrm{PQM}$ surface} if there is an isomorphism $\End(A_{\bar{F}})\simeq \calO$. 
$\calO$-$\mathrm{PQM}$ surfaces over $\Q$ are the central object of interest in this paper, but some of our results apply to abelian surfaces whose geometric endomorphism ring is a possibly non-maximal order in a non-split quaternion algebra. We call such surfaces simply \defi{PQM surfaces}.

We emphasize that this is a restrictive definition of `PQM': we require that $\End(A_{\Fbar})$ does not merely contain such a quaternion order, but is equal to it. In particular, under our terminology, a $\mathrm{PQM}$ surface $A$ is geometrically simple.

Concerning actions: we will use view Galois actions as \emph{right} actions. We will view $\End(A)$ as acting on $A$ on the \emph{left}. 
If a group $G$ acts on a set $X$ on the right, we write $X^G$ for the set of $G$-fixed points.

\section{Quaternionic arithmetic}\label{section: quaternionic algebra}

This section collects some algebraic calculations in the quaternion order $\calO$. 
It can be safely skipped on a first pass; the reader can return back to it when these calculations are used.

\subsection{The normalizer of a maximal order}\label{subsec:normalisers}

We recall the following characterization of the normalizer $N_{B^{\times}}(\calO)$ of $\calO$ in $B^{\times}$ (with respect to the conjugation action).

\begin{lemma}\label{lemma: description normalizer maximal order}
An element of $B^{\times}/\Q^{\times}$ lies in $N_{B^{\times}}(\calO)/\Q^{\times}$ if and only if it can be represented by an element of $\calO$ of reduced norm dividing $\disc(B)$.
\end{lemma}
\begin{proof}
An element $b\in B$ lies in $N_{B^{\times}}(\calO)$ if and only if it lies in the local normalizer $N_{(B\otimes \Q_p)^{\times}}(\calO\otimes \Z_p)$ for all primes $p$.
If $p$ does not divide $\disc(B)$, then this normalizer group equals $\Q_p^{\times}(\calO\otimes \Z_p)^{\times}$ \cite[(23.2.4)]{VoightQA}.
If $p$ divides $\disc(B)$, this group equals $(B\otimes \Q_p)^{\times}$ ((23.2.8) in op. cit.).
If $b$ has norm dividing $\disc(B)$, then this description shows that $b$ lies in all local normalizer groups. 
Conversely, if $b$ normalizes $\calO$ then this description shows that there exists a finite adele $(\lambda_p)_p$ such that $\lambda_p b \in (\calO\otimes\Z_p)^{\times}$ for all $p\nmid \disc(B)$ and such that $\nrd(\lambda_p b)$ has $p$-adic valuation $\leq 1$ for all $p\mid \disc(B)$.
Since $\Z$ has class number one, there exists $\lambda\in \Q^{\times}$ such that $\lambda\lambda_p^{-1}\in \Z_p^{\times}$ for all $p$ and so $\lambda b \in \calO$ has norm dividing $\disc(B)$, as desired.
\end{proof}

We recall for future reference that the quotient of $N_{B^{\times}}(\calO)/\Q^{\times}$ by the subgroup $\calO^{\times}/\{\pm 1\}$ is by definition the Atkin--Lehner group $W$ of $\calO$, an elementary abelian $2$-group whose elements can be identified with positive divisors of $\disc(B)$.

\subsection{Dihedral actions on \texorpdfstring{$\calO$}{O}}\label{subsection: dihedral actions on calO}
For reasons that will become clear in \S\ref{subsection: endomorphism field}, we are interested in subgroups $G\subset \Aut(\calO)$
isomorphic to $D_n$ for some $n\in \{1,2,3,4,6\}$. 
In this section we describe these subgroups very explicitly.

By the Skolem--Noether theorem, every ring automorphism of $\calO$ is of the form $x\mapsto b^{-1}xb$ for some $b\in B^{\times}$ normalising $\calO$, and $b$ is uniquely determined up to $\Q^{\times}$-multiples. 
Therefore $\Aut(\calO) \simeq N_{B^{\times}}(\calO)/\Q^{\times}$.

If $b\in B^{\times}$, write $[b]$ for its class in $B^{\times}/\Q^{\times}$.
\begin{lemma}
\label{lemma: centraliser of order 2 subgroup of automorphism group}
Every element of $N_{B^{\times}}(\calO)/\Q^{\times}$ of order $2$ is represented by an element $b\in \calO$ such that $b^2=m\neq 1$ is an integer dividing $\disc(B)$.
Moreover $\calO^{\langle b\rangle} = \{x\in \calO \mid b^{-1}xb = x \}$ is isomorphic to an order in $\Q(\sqrt{m})$ containing $\Z[\sqrt{m}]$.
\end{lemma}
\begin{proof}
By Lemma \ref{lemma: description normalizer maximal order}, we may choose a representative $b\in N_{B^{\times}}(\calO)$ lying in $\calO$ whose norm $\nrd(b)$ divides $\disc(B)$. 
Since the element has order $2$, $m := b^2 = -\nrd(b)$ is an integer. 
We have $m\neq 1$ since otherwise $b^2=1$ hence $b =\pm 1\in \Q^{\times}$, which is trivial in $N_{B^{\times}}(\calO)/\Q^{\times}$.
This implies $\calO^{\langle b\rangle} = \{x\in B\mid b^{-1}xb = x\}$ is an order in $B^{\langle g\rangle } = \Q(b)$ containing $\Z[b]\simeq\Z[\sqrt{m}]$, as claimed.
\end{proof}

\begin{lemma}\label{lemma: D2 subgroup Aut(O)}
    Let $G\subset N_{B^{\times}}(\calO)/\Q^{\times}$ be a subgroup isomorphic to $D_2 = C_2\times C_2$.
    Then there exist elements $i,j,k\in \calO$ such that $B$ has basis $\{1,i,j,k\}$, such that $i^2 = m$, $j^2 =n$ and $k^2 = t$ all divide $\disc(B)$, such that $ij = -ji$ and $ij \in \Q^{\times} k$, and such that $G  = \{1,[i],[j],[k]\}$.
    Moreover, $t$ equals $-mn$ up to squares.
\end{lemma}
\begin{proof}
    By Lemma \ref{lemma: centraliser of order 2 subgroup of automorphism group}, we can pick representatives $i,j,k \in \calO$ of the nontrivial elements of $G$ that each square to an integer dividing $\disc(B)$.
    Since $G$ is commutative, $ji = \lambda ij$ for some $\lambda \in \Q^{\times}$.
    Comparing norms shows that $\lambda=\pm 1$.
    If $\lambda=1$, then $ij = ji$ but this would imply that $B$ is commutative, contradiction.
    Therefore $ij  = -ji$.
    Finally, since $[i][j] = [k]$, $k \in \Q^{\times} ij$. Taking norms, we see that $t$ equals $-mn$ up to squares.
\end{proof}

\begin{lemma}\label{lemma: D4 subgroup Aut(O)}
    Let $G\subset N_{B^{\times}}(\calO)/\Q^{\times}$ be a subgroup isomorphic to $D_4$.
    Then there exists elements $i,j\in \calO$ such that $B$ has basis $\{1,i,j,ij\}$, such that $i^2= -1$, $j^2= m$ divides $\disc(B)$ and $ij=-ji$, and such that $G = \langle [1+i],[j]\rangle$.
    Moreover, $2\mid \disc(B)$.
\end{lemma}
\begin{proof}
    The fact that such $i,j\in B$ exist follows from \cite[\S32.5 and \S32.6]{VoightQA} (itself based on results of \cite{ChinburgFriedman-finitesubgroupskleinian}).
    By $\Q^{\times}$-scaling $j$ we may assume that $j^2 = m$ is a squarefree integer.
    Since $1+i,j \in N_{B^{\times}}(\calO)$, Lemma \ref{lemma: description normalizer maximal order} shows that $i,j\in \calO$ and $m\mid \disc(B)$ and $\nrd(1+i) = 2\mid \disc(B)$.
\end{proof}

\begin{lemma}\label{lemma: D3,D6 subgroup Aut(O)}
    Let $G\subset N_{B^{\times}}(\calO)/\Q^{\times}$ be a subgroup isomorphic to $D_3$ or $D_6$.
    Then there exist elements $\omega, j \in \calO$ such that $B$ has basis $\{1,\omega,j,\omega j\}$, such that $\omega^3=1$, $j^2= m \mid \disc(B)$ and $\omega j = j\bar{\omega}= j(-1-\omega)$, and such that $G = \langle [1+\omega],[j]\rangle$ if $G \simeq D_3$ and $G = \langle [1-\omega], [j]\rangle$ if $G\simeq D_6$.
    Moroever, if $G \simeq D_6$ then $3\mid \disc(B)$.
\end{lemma}
\begin{proof}
    Identical to that of Lemma \ref{lemma: D4 subgroup Aut(O)}, again using \cite[\S32.5 and \S32.6]{VoightQA} and Lemma \ref{lemma: description normalizer maximal order}.
\end{proof}

\subsection{Fixed point subgroups modulo \texorpdfstring{$N$}{N}}

For reasons similar to those of \S\ref{subsection: dihedral actions on calO}, we study the fixed points of $G$-actions on $\calO/N\calO$ for subgroups $G\subset \Aut(\calO)$ isomorphic to $D_n$ for some $n\in \{1,2,3,4,6\}$ and integers $N\geq 1$ of interest.

\begin{theorem}\label{theorem: fixed points of O/N}
Let $G$ be a subgroup of $\Aut(\calO)$ isomorphic to $D_n$ for some $n\in \{1,2,3,4,6\}$.
\begin{enumalph}
\item Suppose that $N$ is coprime to $2$ and $3$. Then $(\calO/N\calO)^G$ is isomorphic to $(\Z/N\Z)^2$ if $G=D_1$ and isomorphic to $\Z/N\Z$ if $G=D_2,D_3,D_4$ or $D_6$.
\item The group $(\calO/3\calO)^G$ is isomorphic to $(\Z/3\Z)^2$ if $G=D_1$; isomorphic to $\Z/3\Z$ if $G=D_2,D_4,D_6$; and isomorphic to $\Z/3\Z$ or $(\Z/3\Z)^2$ if $G = D_3$.
\item We have
\begin{align*}
(\calO/2\calO)^G \simeq 
\begin{cases}
(\Z/2\Z)^2,(\Z/2\Z)^3\text{ or }(\Z/2\Z)^4  & \text{ if } G = D_1,\\
(\Z/2\Z)^2 \text{ or }(\Z/2\Z)^3  & \text{ if } G = D_2,\\
(\Z/2\Z)^2  & \text{ if } G = D_4,\\
\Z/2\Z & \text{ if } G =D_3\text{ or }D_6.
\end{cases}
\end{align*}
\end{enumalph}
\end{theorem}
\begin{proof}
The reduction map $r_N\colon \calO^G\otimes \Z/N\Z \rightarrow (\calO/N\calO)^G$ is injective and its cokernel is isomorphic to the $N$-torsion of the group cohomology $H^1(G,\calO)$. Indeed, this can be seen by taking $G$-fixed points of the exact sequence $0\rightarrow \calO \rightarrow \calO \rightarrow \calO/N \rightarrow 0$.
The group $\calO^G$ is isomorphic to $\Z^2$ if $G = D_1$ and to $\Z$ if $G = D_2,D_3,D_4$ or $D_6$.
Since the finite abelian group $H^1(G,\calO)$ is killed by the order of $G$, Part (a) immediately follows.
To prove (b) and (c), it therefore suffices to prove that $H^1(G,\calO)[6]$ is a subgroup of $(\Z/2\Z)^2$ if $G=D_1$; isomorphic to $(\Z/2\Z)$ or $(\Z/2\Z)^2$ if $G=D_2$; a subgroup of $\Z/3\Z$ if $G=D_3$; isomorphic to $(\Z/2\Z)$ if $G=D_4$; and trivial if $G=D_6$.
Since $H^1(G,\calO\otimes \Z_p)\simeq H^1(G,\calO)\otimes \Z_p$ for all primes $p$ and since $\Aut(\calO\otimes \Z_p)$ has only finitely many subgroups isomorphic to $G$ up to conjugacy, this is in principle a finite computation; we give a more detailed proof below.

\underline{Case $G=D_1$.}
Since $G=D_1=C_2$ has order $2$, $H^1(G,\calO)$ is $2$-torsion and is isomorphic to the cokernel of $r_2\colon \calO^G\otimes \Z/2\Z\rightarrow (\calO/2\calO)^G$. 
By Lemma \ref{lemma: centraliser of order 2 subgroup of automorphism group}, $\calO^G\simeq \Z^2$ and so this cokernel is either $0, \Z/2\Z$ or $(\Z/2\Z)^2$. It follows that $H^1(G,\calO)\simeq 0, \Z/2\Z$ or $(\Z/2\Z)^2$. 

\underline{Case $G=D_2$.}
By Lemma \ref{lemma: D2 subgroup Aut(O)}, there exist $i,j,k\in \calO$ such that $i^2=m$, $j^2=n$ and $k^2=t$ are all integers dividing $\disc(B)$, such that $ij=-ji$ and $k \in \Q^{\times}ij$ and such that $G = \{1,[i],[j],[k]\}$.
Let $S_i = \calO \cap \Q(i)$, $S_j = \calO\cap \Q(j)$, $S_k = \calO \cap \Q(k)$. Then $S_i$ is an order in $\Q(i)$ containing $\Z[i]$, and similarly for $S_j$ and $S_k$.
Since $-mn$ equals $t$ up to squares, upon reordering $\{i,j,k\}$ we may assume that $\Z[i]$ is maximal at $2$.
Therefore $\Z[\sqrt{m}]\otimes (\Z/2\Z) = S_i\otimes (\Z/2\Z)\subset (\calO/2\calO)$ is a subring on which $G$ acts trivially. It follows that $(\Z/2\Z)^2\subset (\calO/2\calO)^G$.
We will now show that $G$ acts nontrivially on $(\calO/2\calO)$, so assume by contradiction that this action is trivial.
By the classification of involutions on finite free $\Z$-modules, every such involution is a direct sum of involutions of the form $x\mapsto x$, $x\mapsto -x$ and $(x,y)\mapsto (y,x)$.
If $G = \langle [i],[j]\rangle$ acts trivially on $\calO/2\calO$, then both $[i],[j]\in \Aut(\calO)$ are direct sums of involutions of the first two kinds.
It follows that $\calO$ is a direct sum of the eigenspaces corresponding to the eigenvalues of $[i]$ and $[j]$.
It follows that $\calO = \Z1\oplus \Z i \oplus \Z j \oplus \Z k$.
This implies that the discriminant of $\calO$ is $\pm 4u$, contradicting the fact that $\calO$ is maximal at $2$.
We conclude that $(\calO/2\calO)^G$ is $(\Z/2\Z)^2$ or $(\Z/2\Z)^3$ and since $G$ is coprime to $3$, this proves that $H^1(G,\calO)[6]$ is isomorphic to $\Z/2\Z$ or $(\Z/2\Z)^2$. 


\underline{Case $G=D_4$.}
Let $i,j\in \calO$ be elements satisfying the conclusion of Lemma \ref{lemma: D4 subgroup Aut(O)}, so $G = \langle [1+i],[j]\rangle$.
Since $\Z[i]$ is maximal at $2$, the map $\Z[i]\otimes \Z/2\Z\rightarrow \calO/2\calO$ is injective. 
Since $G$ acts trivially on the image of this map, $(\calO/2\calO)^G$ contains $(\Z/2\Z)^2$.
We need to show that $(\calO/2\calO)^G = (\Z/2\Z)^2$.
To prove this, it is enough to show that $(\calO/2\calO)^{\langle 1+i \rangle} = (\Z/2\Z)^2$.
Since $\calO$ is ramified at $2$, there exists a unique conjugacy class of embeddings $\Z_2[i]\hookrightarrow \calO_{\Z_2}$ \cite[Proposition 30.5.3]{VoightQA}.
Therefore it is enough to verify that $(\calO/2\calO)^{\langle 1+i \rangle} = (\Z/2\Z)^2$ in a single example, for which this can be checked explicitly.
Indeed, one may take $B = \left(\frac{-1,6}{\Q}\right)$, which has maximal order with $\Z$-basis $\{1,(1+i+ij)/2,(1-i+ij)/2,(j+ij)/2\}$.
Since $\# G$ is coprime to $3$, we conclude that $H^1(G,\calO)[6] = \Z/2\Z$.

\underline{Case: $G=D_3,D_6$.}
Let $\omega,j\in \calO$ be elements satisfying the conclusion of Lemma \ref{lemma: D3,D6 subgroup Aut(O)}.
Let $C_n\leq D_n$ be the cyclic normal subgroup of order $n$ for $n\in \{3,6\}$.
The low terms of the Lyndon–Hochschild–Serre spectral sequence give rise to the exact sequence
\begin{align}\label{equation: LHS spectral sequence}
 0\rightarrow H^1(C_2,\calO^{C_n})\rightarrow H^1(G,\calO) \rightarrow H^1(C_n,\calO)^{C_2}\rightarrow H^2(C_2,\calO^{C_n}).   
\end{align}
The subring $\calO^{C_n}$ equals $\calO^{\langle 1\pm \omega\rangle} = \Z[\omega]$ and $C_2=D_n/C_n$ acts on $\calO^{C_n}$ via conjugation $\omega\mapsto \bar{\omega}$.
A cyclic group cohomology calculation shows that $H^i(C_2,\Z[\omega])$ is trivial for all $i\geq 1$.
Therefore $H^1(G,\calO)\simeq H^1(C_n,\calO)^{C_2}$.
Assume $G=D_6$.
Using the analogous exact sequence to \eqref{equation: LHS spectral sequence} for the subgroup $C_3\leq C_6$, we get $H^1(C_6,\calO)\simeq H^1(C_3,\calO)^{C_2}$. 
Since $C_2$ acts trivially on $C_3=\{1,g,g^2\}$ and acts as $-1$ on $\{x\in \calO \mid x+gx+g^2x=0\}$, it will act as $-1$ on $H^1(C_3,\calO)\simeq (\Z/3\Z)^r$, so $H^1(C_3,\calO)^{C_2}=0$ and so $H^1(G,\calO)\simeq H^1(C_6,\calO)^{C_2}\subset H^1(C_6,\calO) \simeq H^1(C_3,\calO)^{C_2}$ is zero too in this case.
It remains to consider the case $G=D_3$.
Then $H^1(G,\calO)\simeq H^1(C_3,\calO)^{C_2}$.
Let $g\in C_3$ be a generator, given by conjugating by $1+\omega$.
Let $L =\{x\in \calO \mid x+gx+g^2x=0\}$.
Using the basis $\{1,\omega,j,\omega\}$ of $B$, we see that $L=\calO \cap \Q(\omega)\cdot j$.
Using the explicit description of group cohomology of cyclic groups, $H^1(C_3,\calO)$ is isomorphic to $L/(1-g) \calO$.
Since $(1-g)\calO$ contains $(1-g)L = (1-\omega)L$, $H^1(C_3,\calO)$ is a quotient of $L/(1-\omega)L$.
Since $(1-\omega)^2 L = 3L$ and $L/3L\simeq (\Z/3\Z)^2$, $L/(1-\omega)L$ is of order $3$.
This shows that $H^1(C_3,\calO)=0$ or $\Z/3\Z$, so $H^1(D_3,\calO) = 0$ or $\Z/3\Z$, as claimed. 

\end{proof}

\begin{remark}
    A calculation with the quaternion algebra package in \texttt{Magma} shows that all the possibilities in Theorem \ref{theorem: fixed points of O/N} do indeed occur.
\end{remark}

The next three lemmas give some more precise information about subgroups $G\subset \Aut(\calO)$ for which $(\calO/2\calO)^G \simeq (\Z/2\Z)^3$. 
In these lemmas, we will use the fact that if $2\mid \disc(B)$, there exists a unique ring homomorphism $\calO/2\calO\rightarrow \F_4$, see \cite[Theorem 13.3.11]{VoightQA}.

\begin{lemma}\label{lemma: involution with (z/2)^3 fixed points}
Let $b\in \calO \cap N_{B^{\times}}(\calO)$ be an element with $b^2 = m \mid \disc(B)$ and $m\neq 1$.
Write $F \subset \calO/2\calO$ for the subset centralized by the reduction of $b$ in $\calO/2\calO$.
Then $F \simeq (\Z/2\Z)^3$ if and only if $2\mid \disc(B)$ and $m\equiv 3 \mod 4$.
In that case $F$ equals the subset of elements of $\calO/2\calO$ whose image under the ring homomorphism $\calO/2\calO \rightarrow \F_4$ lands in $\F_2$.
\end{lemma}
\begin{proof}
Suppose $F\simeq (\Z/2\Z)^3$. 
We first show that $2\mid \disc(B)$. 
If not, then $m$ is odd by Lemma \ref{lemma: centraliser of order 2 subgroup of automorphism group}, $\calO/2\calO \simeq \Mat_2(\F_2)$ and $F$ is the fixed points of conjugating by an element of order dividing $2$ in $\GL_2(\F_2)$. 
Since there is only one involution in $\GL_2(\F_2)$ up to conjugacy, which we may calculate has centralizer $(\Z/2\Z)^2$, this shows that $2\mid \disc(B)$. 
We now show that $2\nmid m$. If $2\mid m$, then since $m$ is squarefree $b$ is a $2$-adic uniformizer of $\calO \otimes \Z_2$.
Then there exists an unramified quadratic subring $S\subset \calO\otimes \Z_2$ isomorphic to $\Z_2\left[\frac{-1+\sqrt{-3}}{2}\right]$ such that $\calO \otimes \Z_2 = S + S\cdot b$ \cite[Theorem 13.3.11]{VoightQA}. 
This shows that conjugation by $b$ acts via $x+yb\mapsto \bar{x}+\bar{y}b$.
This map has $4$ fixed points, hence we obtain a contradiction and $m$ is odd. 
It follows that $F$ is given by the fixed points of conjugating by an element of $(\calO/2\calO)^{\times}$. 
This element is trivial if and only if $b\in 1+2\calO$. 
Since $\calO\otimes \Z_2$ consists of all integral elements of $B\otimes \Q_2$ \cite[Proposition 13.3.4]{VoightQA} and since $b\in \calO$, this is equivalent to $(b-1)/2$ being integral at $2$, that is to say to $m \equiv 1\mod 4$.
This proves the forward direction of the lemma. For the other direction, note that $(\calO/2\calO)^{\times}$ (where $\calO$ is ramified at $2$) has a unique involution up to conjugacy, which can be checked to have $(\Z/2\Z)^3$ fixed points in the presentation \eqref{equation: description OB mod p for p ramified}.
\end{proof}

\begin{lemma}\label{lemma: mod 4 quaternion calculation (z/2)^3}
    Let $b\in \calO \cap N_{B^{\times}}(\calO)$ be an element with $b^2 = m \mid \disc(B)$ and $m\neq 1$.
    Suppose that the conjugation action of $b$ on $\calO/2\calO$ has fixed points $\simeq (\Z/2\Z)^3$.
    Then there exists no $x\in \calO/4\calO$ with  $x \equiv 1 \mod 2\calO$ and $b^{-1}xbx = -1$.
\end{lemma}
\begin{proof}
    Suppose that $x\in \calO/4\calO$ is such an element. 
    Let $y = bx$.
    Since $m b^{-1} = b$, multiplying the equation $b^{-1}xbx = -1$ by $m$ shows that $y^2 = -m$ in $\calO/4\calO$.
    By Lemma \ref{lemma: involution with (z/2)^3 fixed points}, $2\mid \disc(B)$ and $m \equiv 3 \mod{4}$, so $y^2=1$ in $\calO/4\calO$.
    Since $x \equiv 1 \mod 2\calO$, $y = bx \equiv b \mod 2\calO$.
    We may therefore write $y = b + 2z$ for some $z\in \calO/4\calO$.
    We compute, in $\calO/4\calO$, that 
    \begin{align*}
        y^2  =(b+2z)(b+2z) = b^2 + 2(bz+zb)+4z^2 = m + 2(bz+zb)  = 3 + 2(bz+zb).
    \end{align*}
    Since $y^2=1$, this shows that $2(bz+zb) = 2$.
    Write $\bar{b}$ and $\bar{z}$ for the mod $2$ reductions of $b$ and $z$.
    Then the above identity implies that
    \begin{align}\label{equation: mod 4 calculation bz+zb}
       \bar{b}\bar{z}+\bar{z}\bar{b}=1. 
    \end{align}
    Since $2$ is ramified in $B$ and $\calO$ is maximal, there exists a surjective ring homomorphism $\lambda\colon \calO/2\calO\rightarrow \F_4$.
    Applying $\lambda$ to \eqref{equation: mod 4 calculation bz+zb} shows that $\lambda(\bar{b})\lambda(\bar{z})+ \lambda(\bar{z})\lambda(\bar{b}) = \lambda(1) = 1$.
    Since $\F_4$ is commutative, the left hand side of this equation also equals  $2\lambda(\bar{b})\lambda(\bar{z})= 0$, which is a contradiction.
\end{proof}

Recall from Lemma \ref{lemma: D2 subgroup Aut(O)} that a subgroup $G\leq N_{B^{\times}}(\calO)$ isomorphic to $C_2\times C_2$ can be generated by elements $i,j\in \calO$ with $ij = -ji$, $i^2 = m$, $j^2=n$ and $m,n\mid \disc(B)$.
\begin{lemma}\label{lemma: C2xC2 with (Z/2)^3 fixed points}
    Let $G\subset N_{B^{\times}}(\calO)$ be a subgroup isomorphic to $C_2\times C_2$.
    Then $(\calO/2\calO)^G \simeq (\Z/2\Z)^3$ if and only if (in the above notation) $2\mid \disc(B)$ and $m,n \equiv 3\mod 4$.
\end{lemma}
\begin{proof}
    Suppose first that $(\calO/2\calO)^G \simeq (\Z/2\Z)^3$.
    Then the conjugation involutions $[i]$ and $[j]$ have both $2^3$ or $2^4$ fixed points on $\calO/2\calO$. At least one of them, say $j$, has $2^3$ fixed points.
    By Lemma \ref{lemma: involution with (z/2)^3 fixed points}, $2\mid \disc(B)$ and $n\equiv 3 \mod 4$.
    If $2\mid m$, then $i$ is a $2$-adic uniformizer and the action of $i$ on $\calO/2\calO$ would have $2^2$ fixed points (by an argument similar to the proof of Lemma \ref{lemma: involution with (z/2)^3 fixed points}).
    So $m$ is odd. 
    If $m\equiv 1 \mod 4$, then the $2$-adic Hilbert symbol of $(m,n)$ is trivial, contradicting the fact that $2\mid \disc(B)$ and $B \simeq \left(\frac{m,n}{\Q}\right)$.
    We conclude that $m\equiv 3 \mod 4$.
    The converse follows from Lemma \ref{lemma: involution with (z/2)^3 fixed points}.
\end{proof}

\section{Galois actions, polarizations and endomorphisms}
\label{section: galois actions, polarizations and endomorphisms}

This section collects some preliminaries concerning the arithmetic of $\mathrm{PQM}$ surfaces.
In particular, we study the Galois action on the endomorphism algebra, the set of polarizations, the torsion points and the interaction between these.
The most important subsection is \S\ref{subsection: endomorphism field}, where the endomorphism field of a $\mathrm{PQM}$ surface is introduced.

\subsection{Abelian surfaces of \texorpdfstring{$\GL_2$-type}{GL2-type}}

Recall that an abelian surface $A$ over a number field $F$ is said to be \defi{of $\GL_2$-type} if $\End^0(A)$ is a quadratic field extension of $\Q$. 
We will show that if $A$ is geometrically simple and $F$ admits a real place, then this field must be real quadratic.
(The geometrically simple hypothesis is necessary; for example, the simple modular abelian surface $J_1(13)$ satisfies $\End^0(J_1(13)) \simeq \Q(\sqrt{-3})$.)
This is well known over $\Q$ (see \cite[Lemma 2.3]{Rotger-whichquaternion}), which suffices for our purposes---but we also give an argument that works over any field contained in $\R$ that might be of independent interest.
(We thank Davide Lombardo for suggesting it.)

\begin{lemma}\label{lemma: Neron severi rank computation real abelian surfaces}
Let $A/\R$ be an abelian surface. 
Then $\rk \, \NS(A) \geq \rk \, \NS(A_{\C})-1$.
\end{lemma}
\begin{proof}
There exists a two-dimensional $\mathbb{R}$-vector space $W$, a lattice $\Lambda\subset W_{\C} := W\otimes \C$ stable under the automorphism $\sigma$ induced by complex conjugation on the second factor, and a complex analytic isomorphism $A(\C) \simeq (W_{\C})/\Lambda$ that intertwines complex conjugation on $A(\C)$ with $\sigma$.
Under this isomorphism, $\NS(A_{\C})$ can be identified with the set of $\Z$-bilinear alternating forms $E\colon \Lambda \times \Lambda \rightarrow \Z$ with the property that the $\mathbb{R}$-linear extension $E_{\mathbb{R}}$ of $E$ to $W_{\C}$ satisfies $E_{\mathbb{R}}(iv,iw) = E_{\mathbb{R}}(v,w)$ for all $v,w\in \Lambda\otimes \mathbb{R}  = W_{\C}$.
By \cite[Chapter IV, Theorem (3.4)]{Silhol-92} such an $E$ lies in $\NS(A)$ if and only if the associated Hermitian form $E_{\mathbb{R}}(iv,w)+iE_{\mathbb{R}}(v,w)$ is $\mathbb{R}$-valued on $W\times W$, that is to say $E_{\mathbb{R}}(W,W) = 0$.
Since the intersection $\Lambda' = \Lambda \cap W$ is a lattice in $W$, the condition $E_{\mathbb{R}}(W,W) = 0$ is equivalent to $E(\Lambda',\Lambda') = 0$.
In conclusion, $\NS(A) = \ker(\NS(A_{\C}) \rightarrow \Hom(\wedge^2(\Lambda'),\Z))$, where the map sends $E$ to its restriction to $\Lambda'\times \Lambda'$.
Since the target of this map is isomorphic to $\Z$, the lemma is proved.
\end{proof}

\begin{prop}\label{proposition: endomorphism rings of absolutely simple abelian surfaces over R}
Let $A/\R$ be a geometrically simple abelian surface. 
Then $\End(A)$ is isomorphic to $\Z$ or an order in a real quadratic field.
\end{prop}

\begin{proof}
By the classification of endomorphism algebras of complex abelian surfaces \cite[Proposition 5.5.7, Exercise 9.10(1) and Exercise 9.10(4)]{BirkenhakeLange-Complexabelianvarieties}, $\End^0(A_{\C})$ is isomorphic to either $\Q$, a real quadratic field, a non-split indefinite quaternion algebra or a quartic CM field.
The proposition is clear in the first two cases, so we may assume that we are in one of the latter two cases.

Since $\End^0(A)$ acts on the $\Q$-homology of $A(\R)^{\circ} \simeq S^1\times S^1$, there is a (nonzero, hence injective) map $\End^0(A) \hookrightarrow \Mat_2(\Q)$.
Since $\End^0(A_{\C})$ does not embed in $\Mat_2(\Q)$, $\End^0(A)\neq \End^0(A_{\C})$ and so $\End^0(A)$ is at most two-dimensional.
It remains to exclude that $\End^0(A)$ is an imaginary quadratic field, so assume for contradiction that this is the case.
If $\End^0(A_{\C})$ is a quaternion algebra, Lemma \ref{lemma: Neron severi rank computation real abelian surfaces} shows that $\rk(\NS(A))\geq 3-1 = 2$, contradicting the fact that $\End^0(A)$ is imaginary quadratic.
If $\End(A_{\C})$ is a quartic CM field $F$, this CM field has at least two quadratic subfields (namely its unique real quadratic subfield and $\End^0(A)$) so it must be a biquadratic extension of $\Q$.
A counting argument then shows that every CM type of $F$ is imprimitive.
This implies \cite[Theorem 3.5]{Lang-ComplexMultiplication} that $A_{\C}$ is not simple.
We again obtain a contradiction and have completed all cases of the proof.
\end{proof}

\subsection{The endomorphism field of a \texorpdfstring{$\mathrm{PQM}$}{PQM} surface}\label{subsection: endomorphism field}

Let $F$ be a field of characteristic zero and $A/F$ a $\mathrm{PQM}$ surface. 
The absolute Galois group $\Gal_{F}$ acts on $\End(A_{\Fbar})$ on the right by ring automorphisms via $\phi^{\sigma}(a) = \phi\left(a^{\sigma^{-1}}\right)^{\sigma}$ for $\sigma \in \Gal_F$, $\phi \in \End(A_{\Fbar})$ and $a\in A(\Fbar)$.
The kernel of this action cuts out a Galois extension $L/F$ over which all the endomorphisms of $A_{\Fbar}$ are defined.
Following \cite{GK17} we call $L$ the \defi{endomorphism field} of $A$.
This determines an injective map $\rho_{\End}\colon \Gal(L/F) \rightarrow \Aut(\End(A_{\Fbar}))$.
We recall the results of \cite{DR04} studying this map which are relevant for our purposes.
Write $C_n$ (resp. $D_n$) for the cyclic (resp. dihedral) group of order $n$ (resp. $2n$).
Note the isomorphisms $D_1 \simeq C_2$ and $D_2 \simeq C_2\times C_2$.

\begin{prop}\label{proposition: galois group endomorphism field is dihedral if real place}
Let $A/F$ be a $\mathrm{PQM}$ surface with endomorphism field $L$ and let $G = \Gal(L/F)$.
Then $G \simeq C_n$ or $D_n$ for some $n\in \{1,2,3,4,6\}$.
If $F$ admits an embedding into $\mathbb{R}$, then $G\simeq D_n$ for some $n\in \{1,2,3,4,6\}$. 
\end{prop}
\begin{proof}
The classification of finite subgroups of $B^{\times}/\Q^{\times}$ shows that $G$ is isomorphic to $C_n$ or $D_n$ for some $n\in \{1,2,3,4,6\}$ \cite[Proposition 2.1]{DR04}.
It therefore suffices to exclude that $G$ is isomorphic to $C_1, C_3,C_4$ or $C_6$ if there exists an embedding $\iota\colon F\hookrightarrow \mathbb{R}$.
If $G$ is isomorphic to one of these groups, then $\End^0(A)$ is isomorphic to $B$ (if $G$ is trivial) or an imaginary quadratic field \cite[Theorem 3.4(C)]{DR04}. This contradicts Proposition \ref{proposition: endomorphism rings of absolutely simple abelian surfaces over R}.
\end{proof}
\begin{lemma}\label{lemma: A GL2-type iff endo field quadratic}
Let $A$ be a $\mathrm{PQM}$ surface over a number field $F$ admitting a real place. Then $A$ is of $\GL_2$-type if and only if the endomorphism field $L/F$ is a quadratic extension.
\end{lemma}
\begin{proof}
By Proposition \ref{proposition: endomorphism rings of absolutely simple abelian surfaces over R}, $A$ is of $\GL_2$-type if and only if $\End(A) \neq \Z$.
By \cite[Theorem 3.4(C)]{DR04}, $\End(A) \neq \Z$ if and only if $L$ is a cyclic extension of $F$.
By Proposition \ref{proposition: galois group endomorphism field is dihedral if real place}, $L/F$ is cyclic if and only if it is a quadratic extension.
\end{proof}

Assume now that $A$ is an $\calO$-$\mathrm{PQM}$ surface and fix an isomorphism $\End(A_{\Fbar})\simeq \calO$.
By the Skolem--Noether theorem, every ring automorphism of $\calO$ is of the form $x\mapsto b^{-1}xb$ for some $b\in B^{\times}$ normalising $\calO$, and $b$ is uniquely determined up to $\Q^{\times}$-multiples. 
Therefore $\Aut(\calO) \simeq N_{B^{\times}}(\calO)/\Q^{\times} \subset B^{\times}/\Q^{\times}$, hence the map $\Gal(L/F) \rightarrow \Aut(\End(A_{\Fbar}))$ can be viewed as an injective homomorphism
\begin{align}\label{equation: endomorphism field homomorphism}
\rho_{\End}\colon \Gal(L/F) \rightarrow \Aut(\calO) \simeq N_{B^{\times}}(\calO)/\Q^{\times}
\end{align}
whose image is isomorphic to $C_n$ or $D_n$ for some $n\in \{1,2,3,4,6\}$ by Proposition \ref{proposition: galois group endomorphism field is dihedral if real place}.

\begin{remark}
    The existence of a polarization of a certain type puts restrictions on the Galois group of the endomorphism field, see \cite[Theorem 3.4]{DR04}.
    In particular, that theorem shows that if an $\calO$-$\mathrm{PQM}$ surface $A$ is principally polarized over $F$ then this Galois group is $\{1\}$, $C_2$ or $C_2\times C_2$.
\end{remark}

For future reference we record the following result of Silverberg \cite[Proposition 2.2]{Silverberg92a}.
\begin{prop}[Silverberg]\label{prop: Silverberg result endo field}
    Let $N\geq 3$ be an integer and suppose that the $\Gal_F$-action on $A[N]$ is trivial. 
    Then $L = F$.
\end{prop}

We also record the useful fact that the endomorphism field is preserved by quadratic twist.
\begin{lemma}\label{lemma: quadratic twisting doesnt change endo field}
    Let $A/F$ be a $\mathrm{PQM}$ surface and $M/F$ a quadratic extension.
    Let $A^M$ be the quadratic twist of $A$ along $M/F$.
    Then under the identification $\End(A_{\Fbar}) = \End((A^M)_{\Fbar})$, $\rho_{\End,A} = \rho_{\End, A^M}$.
\end{lemma}
\begin{proof}
    This follows from the fact that $-1$ is central in $\End(A_{\Fbar})$.
\end{proof}

\subsection{Polarizations and positive involutions}\label{positive involutions}

Let $A$ be an abelian surface over a field $F$ of characteristic zero.
Recall that a polarization is an ample class $L$ in $\NS(A)$. 
Such a class gives rise to an isogeny $\lambda_L\colon A\rightarrow A^{\vee}$, and we frequently identify $L$ with this isogeny. 
There exists unique positive integers $d_1\mid d_2$ such that $\ker(\lambda_L)(\Fbar)\simeq (\Z/d_1)^2\times (\Z/d_2)^2$; the pair $(d_1,d_2)$ is called the \defi{type} of the polarization and the integer $\deg(L) = d_1d_2$ is called its \defi{degree}.
We say two polarizations $L$ and $L'$ are $\Q^{\times}$-equivalent if there exist nonzero integers $m,n$ such that $mL = nL'$, and we call a $\Q^{\times}$-equivalence class of polarizations a $\Q^{\times}$-polarization. 
Every $\Q^{\times}$-polarization contains a unique polarization of type $(1,d)$ for some $d\geq 1$.

Recall that a positive involution of $B$ is a $\Q$-linear involution $\iota \colon B\rightarrow B$ satisfying $\iota(ab) = \iota(b)\iota(a)$ and $\trd(a\iota(a)) \in \Q_{\geq 0}$ for all $a,b\in B$.
By the Skolem--Noether theorem, every such involution is of the form $b\mapsto \mu^{-1} \bar{b}\mu$, where $\bar{b}=\trd(b)-b$ denotes the canonical involution and $\mu \in B^{\times}$ is an element with $\mu^2 \in \Q_{<0}$. 
Two such elements $\mu,\mu'\in B^{\times}$ give rise to the same involution if and only if $\mu$ is a $\Q^{\times}$-multiple of $\mu'$.

To combine these two notions, suppose that $\End(A) = \End(A_{\Fbar})\simeq \calO$; let us fix such an isomorphism to identify $\End(A)$ with $\calO$.
Given a polarization $L$ of $A$, the Rosati involution on $\End^0(A)$, defined by $f \mapsto \lambda_L^{-1}\circ  f^{\vee} \circ \lambda_L$, corresponds to a positive involution $\iota_L$ of $B$.

\begin{prop}\label{prop:relation polarizations positive involutions}
The assignment $L\mapsto \iota_L$ induces a bijection between the set of $\Q^{\times}$-polarizations of $A$ and the set of positive involutions of $B$.
In addition, if $L$ is a polarization and $\mu\in B^{\times}$ is an element such that $\iota_L$ is of the form $b\mapsto \mu^{-1}\bar{b}\mu$, then 
\begin{align}\label{equation: degree polarization vs norm positive involution}
   \deg(L) \equiv \disc(B)\cdot \nrd(\mu) \mod \Q^{\times 2}. 
\end{align}
\end{prop}
\begin{proof}
This can be deduced from \cite[Theorem 3.1]{DR04}, but can also be proved purely algebraically as follows.
Choose an element $\nu\in \calO$ with $\nu^2 = -\disc(B)$.
Then it is well known \cite[Lemma 43.6.23]{VoightQA} that $A$ has a unique principal polarization $M$ such that $\iota_M(b) = \nu^{-1} \bar{b} \nu$ for all $b\in B$.
To determine all polarizations of $A$, consider the maps
\begin{align*}
    (\NS(A)\otimes\Q)\setminus \{0\} \xrightarrow{\alpha} \{x\in B^{\times} \mid \nu^{-1} \bar{x} \nu  = x\} \xrightarrow{\beta} \{\mu \in B^{\times} \mid \bar{\mu} = -\mu\},
\end{align*}
where $\alpha(L)=\lambda_M^{-1}\circ\lambda_L $ and $\beta(x) = \nu x$.
Since $L\mapsto \lambda_L$ induces a bijection $\NS(A)\otimes \Q \rightarrow \{f\in \Hom(A,A^{\vee})\mid f^{\vee}=f\}$, $\alpha$ is a bijection. Moreover, $\beta$ is a bijection by a direct computation. 
In addition, one can also compute that the Rosati involution associated to a Neron--Severi class $L$ is given by conjugation by $\beta(\alpha(L))$.
Both $(\NS(A)\otimes\Q)\setminus \{0\}$ and $\{\mu \in B^{\times} \mid \bar{\mu} = -\mu\}$ have evident $\Q^{\times}$-actions, and their quotients are given by the set of $\Q^{\times}$-polarizations and the set of positive involutions on $B$ respectively.
Combining these observations shows that $L\mapsto \iota_L$ is indeed a bijection between the set of $\Q^{\times}$-polarizations and the set of positive involutions.
To check \eqref{equation: degree polarization vs norm positive involution}, we compute that for $\alpha(L) = x$ and $\mu = \nu x$: $\deg(L) = \nrd(x) = \nrd(\mu)/\nrd(\nu) \equiv \disc(B) \cdot \nrd(\mu) \mod \Q^{\times 2}$.
\end{proof}

\begin{remark}
    If we want to avoid choosing an isomorphism $\End(A)\simeq \calO$, we may rephrase Proposition \ref{prop:relation polarizations positive involutions} as saying that there is a bijection between $\Q^{\times}$-polarizations on $A$ and positive involutions on the quaternion algebra $\End^0(A)$.
\end{remark}

Now suppose that $A/F$ is an abelian surface with $\End(A_{\Fbar})\simeq \calO$.
Recall from \S\ref{subsection: endomorphism field} that $\Gal_F$ acts on $\End(A_{\Fbar})$ by ring automorphisms.
If $L$ is a polarization on $A_{\Fbar}$, the Rosati involution associated to $L$ is of the form $b\mapsto \mu^{-1}b\mu$ for some $\mu \in \End^0(A_{\Fbar})$, uniquely determined up to $\Q^{\times}$-multiple..
Therefore the imaginary quadratic field $\Q(\mu)\subset \End^0(A_{\Fbar})$ is independent of the choice of $\mu$.
\begin{corollary}\label{cor:quadsubring via polarization}
    The map $L\mapsto \Q(\mu)$ constructed above induces a bijection between $\Q^{\times}$-polarizations of $A_{\Fbar}$ and imaginary quadratic fields contained in $\End^0(A_{\Fbar})$.
    A polarization descends to $A$ if and only if the imaginary quadratic field is $\Gal_F$-normalized. 
\end{corollary}

\begin{proof}
The bijection part immediately follows from Proposition \ref{prop:relation polarizations positive involutions}, together with the fact that the set of positive involutions on $\End^0(A_{\Fbar})$ is in bijection with the set of imaginary quadratic subfields of $\End^0(A_{\Fbar})$.

Since taking the Rosati involution is $\Gal_F$-equivariant, this bijection preserves the Galois action on both sides.
This induces a bijection on the $\Gal_F$-fixed points on both sides, justifying the last sentence of the corollary.
\end{proof}

\subsection{The distinguished quadratic subring}
If $A/\Q$ is an $\calO$-$\mathrm{PQM}$ surface of $\GL_2$-type, then the torsion groups $A[n]$ are modules over $S/nS$, where $S$ is the real quadratic ring $\End(A)$. If $A$ is not of $\GL_2$-type, then $\End(A) = \Z$, and so it may seem that there is no structure to exploit. However, we have seen in Corollary \ref{cor:quadsubring via polarization} that any polarization of $A$ determines a $\Gal_{\Q}$-stable imaginary quadratic subring $S \subset \End(A_{\Qbar})$.

\begin{definition}\label{definition: distinguished quadratic subring}
Let $A/\Q$ be an $\calO$-$\mathrm{PQM}$ surface.
If $A$ is of $\GL_2$-type, let $M = \End^0(A)$.
Otherwise, let $M\subset \End^0(A_{\Qbar})$ be the imaginary quadratic field corresponding to the unique primitive polarization on $A$ via Corollary \ref{cor:quadsubring via polarization}.
We call $M\subset \End^0(A_{\Qbar})$ the \defi{distinguished quadratic} subfield and $S = M \cap \End(A_{\Qbar})$ the \defi{distinguished quadratic subring} of $A$. 
\end{definition}

The next proposition describes the distinguished quadratic subring more explicitly.

\begin{prop}\label{proposition: existence normalised torus}
Let $A/\Q$ be an $\calO$-$\mathrm{PQM}$ surface and let $S$ be its distinguished quadratic subring, seen as a subring of $\calO$ using an isomorphism $\calO \simeq \End(A_{\Qbar})$. 
Let $G$ be the Galois group of the endomorphism field of $A$ (as in \S\ref{subsection: endomorphism field}).
\begin{enumalph}
    \item $S$ is isomorphic to an order of $\Q(\sqrt{m})$ containing $\Z[\sqrt{m}]$ for some $m\in \Z_{\geq 2}$ dividing $\disc(B)$ if $G=C_2$; to an order of $\Q(\sqrt{-m})$ containing $\Z[\sqrt{-m}]$ for some $m\in \Z_{\geq 2}$ dividing $\disc(B)$ if $G = D_2$; to $\Z[i]$ with $i^2=-1$ if $G= D_4$; and to $\Z[\omega]$ with $\omega^3 = 1$ if $G = D_3$ or $D_6$.
    \item $S$ is an order in a quadratic field, maximal away from $2$ and unramified away from $6\disc(B)$.
\end{enumalph}
\end{prop}
\begin{proof}
The description of $S$ in the $C_2$ case follows from Lemma \ref{lemma: centraliser of order 2 subgroup of automorphism group}.
If $G\not\simeq C_2$ (in other words, if $A$ is not of $\GL_2$-type by Lemma \ref{lemma: A GL2-type iff endo field quadratic}), then Corollary \ref{cor:quadsubring via polarization} shows that $S$ is the unique imaginary quadratic subring of $\End(A_{\Qbar})$ that is $\Gal_{\Q}$-stable and that is optimally embedded, i.e. $(S\otimes \Q) \cap \calO = S$.
So to prove (a) it suffices to find a subring of $\calO$ satisfying the stated conditions. 
This follows from the explicit description of the $G$-action given in \S\ref{subsection: dihedral actions on calO}.
Part (b) immediately follows from the first part. 
\end{proof}

\subsection{The enhanced Galois representation}\label{enhanced representation}

Let $A$ be an $\calO$-$\mathrm{PQM}$ surface over a field $F$ of characteristic zero, and fix an isomorphism $\calO\simeq \End(A_{\Fbar})$ so that $\calO$ acts on $A_{\Fbar}$ on the left.
In \S\ref{subsection: endomorphism field} we have described how $\Gal_F$ acts on the endomorphism ring $\calO$; this action is encoded by the homomorphism $\rho_{\End}\colon \Gal_F\rightarrow \Aut(\calO)$ of Equation \ref{equation: endomorphism field homomorphism}.
On the other hand $\Gal_F$ acts on the torsion points of $A_{\Fbar}$.
In this section we formalize the interaction of these two $\Gal_F$-actions using a homomorphism that we call the \defi{enhanced Galois representation}.
This basic definition might be of independent interest and will be used in the proof of Theorem \ref{thm:gl2type classification}, more specifically to exclude $(\Z/2\Z)^3$ in the $\GL_2$-type case in Proposition \ref{prop: GL2 type no (Z/2)^3 torsion}.

Let $I\subset \calO$ be a $\Gal_F$-stable two-sided ideal, for example $I = N\cdot \calO$ for some integer $N\geq 1$. 
The subgroup $A[I](\Fbar)\subset A(\Fbar)$ of points killed by $I$ is a $\Gal_F$-module. 
Let $\GL(A[I])$ be the group of $\Z$-module automorphisms of $A[I](\Fbar)$, seen as acting on $A[I](\Fbar)$ on the right.
The $\Gal_F$-action on $A[I]$ is encoded in a homomorphism $\rho_I\colon \Gal_F\rightarrow \GL(A[I])$.
The left $\calO$-action on $A_{\Fbar}$ induces an $\calO/I$-action on $A[I](\Fbar)$ such that 
\begin{align}\label{equation: compatibility gal action O action}
(a\cdot P)^\sigma = a^\sigma \cdot P^\sigma
\end{align}
for all $P \in A[I](\Fbar), \ a \in \calO$ and $\sigma \in \Gal_F.$
Let $\Aut^{\circ}(A[I])$ be the subgroup of pairs $(\gamma, \varphi) \in \Aut(\calO) \times \GL(A[I])$ such that $(a\cdot P)^{\varphi} = a^{\gamma} \cdot P^{\varphi}$ for all $a\in \calO$ and $P \in A[I](\Fbar)$.
The compatibility \eqref{equation: compatibility gal action O action} implies that the product homomorphism $\rho_{\End}\times \rho_I\colon \Gal_F \rightarrow \Aut(\calO)\times \GL(A[I])$ lands in $\Aut^{\circ}(A[I])$, so we obtain a homomorphism
\begin{align}\label{equation: def abstract enhanced gal rep}
    \rho^{\circ}_I \colon \Gal_F \rightarrow \Aut^{\circ}(A[I]).
\end{align}
We now identity $\Aut^{\circ}(A[I])$ with an explicit semidirect product.
Consider the group $\Aut(\calO)\ltimes (\calO/I)^{\times}$, where $\Aut(\calO)$ acts on $(\calO/I)^{\times}$ via restricting the standard right $\Aut(\calO)$-action on $\calO/I$ to $(\calO/I)^{\times}$. 
Multiplication in this group is given by $(\gamma_1,x_1) \cdot (\gamma_2,x_2) = (\gamma_1\gamma_2, x_1^{\gamma_2}x_2)$.
The $\calO/I$-module $A[I](\Fbar)$ is free of rank $1$ \cite{Ohta74}.
Let $Q \in A[I](\Fbar)$ be an $\calO/I$-module generator.
For every $(\gamma,x) \in 
\Aut(\calO)\ltimes (\calO/I)^{\times}$, let $\varphi_{(\gamma,x)}$ be the element of $\GL(A[I])$ sending $a\cdot Q$ to $a^{\gamma} x\cdot Q$ for all $a\in \calO/I$.
\begin{lemma}\label{lemma: explicit description AutO after choice of generator}
    The map $(\gamma, x)\mapsto (\gamma, \varphi_{(\gamma,x)})$ induces an isomorphism $\Aut(\calO)\ltimes (\calO/I)^{\times}\xrightarrow{\sim} \Aut^{\circ}(A[I])$.
\end{lemma}
\begin{proof}
    This is a formal verification.
    The inverse of this isomorphism is given by sending $(\gamma, \varphi)$ to $(\gamma, x)$, where $x\in (\calO/I)^{\times}$ is the unique element with $Q^{\varphi} = x\cdot Q$.
\end{proof}
Using Lemma \ref{lemma: explicit description AutO after choice of generator}, we may view the homomorphism \eqref{equation: def abstract enhanced gal rep} as a homomorphism 
\begin{align}\label{equation: def concrete enhanced gal rep}
    \rho^{\circ}_I \colon \Gal_F \rightarrow \Aut(\calO)\ltimes (\calO/I)^{\times}.
\end{align}
\begin{definition}
    The homomorphism \eqref{equation: def abstract enhanced gal rep} or, after a choice of $\calO/I$-module generator of $A[I](\Fbar)$, the homomorphism \eqref{equation: def concrete enhanced gal rep}, is called the \defi{enhanced Galois representation} associated to $A$ and $I$.
\end{definition}


Since $\Aut^{\circ}(A[I])$ is a subgroup of $\Aut(\calO)\times \GL(A[I])$, it comes equipped with projection homomorphisms $\pi_1\colon \Aut^{\circ}(A[I])\rightarrow \Aut(\calO)$ and $\pi_2\colon \Aut^{\circ}(A[I])\rightarrow \GL(A[I])$ satisfying $\rho_{\End} = \pi_1 \circ \rho^{\circ}_I$ and $\rho_I = \pi_2 \circ \rho_{I}^{\circ}$.
\begin{remark}
    Suppose that $\rho_{\End}$ is trivial, in other words $\End(A) = \End(A_{\Fbar}) \simeq \calO$.
    Then the homomorphism \eqref{equation: def concrete enhanced gal rep} lands in the subgroup $\{1\} \ltimes (\calO/I)^{\times}$ and hence simplifies to a homomorphism $\Gal_F\rightarrow (\calO/I)^{\times}$.
    This recovers the well known description \cite{Ohta74} of the Galois representation $\rho_I$ in this case.
\end{remark}
We show that usually, $\rho^{\circ}_I$ does not contain more information than $\rho_I$ itself, using the following well known lemma.

\begin{lemma}\label{lemma: reduction map is injective}
Let $G$ be a finite subgroup of $\GL_n(\Z)$ for some $n\geq 1$ and let $\red_N\colon G\rightarrow \GL_n(\Z/N\Z)$ be the restriction of the reduction map.
Then $\red_N$ is injective if $N \geq 3$, and every element of the kernel of $\red_2$ has order $1$ or $2$.
\end{lemma}
\begin{proof}
This is a classical result of Minkowski \cite{Minkowski-1887}; see \cite[Theorem 4.1]{SilberbergZarhin95} for an accessible reference.
\end{proof}


\begin{prop}\label{prop: enhanced gal rep iso to normal gal rep if N >2}
    Supppose that $I = N\cdot \calO$ for some integer $N\geq 3$.
    Then $\pi_2$ is injective on the image $\rho_I^{\circ}$.
    Consequently, the image of $\rho_I^{\circ}$ is isomorphic to the image of $\rho_I$.
\end{prop}

\begin{proof}
    Choose a $\calO/N$-module generator $Q\in A[N](\Fbar)$.
    If $(\gamma, \varphi)\in \ker(\pi_2)$, then $\varphi = \text{Id}$ and $a\cdot Q = a^{\gamma}\cdot  Q$ for all $a\in \calO/N$.
    So $a = a^{\gamma}$ for all $a\in \calO/N$.
    Therefore $\gamma\in \ker(\Aut(\calO)\rightarrow \Aut(\calO/N))$.
    By Lemma \ref{lemma: reduction map is injective}, this kernel does not contain any nontrivial element of finite order. 
    However, the image of $\rho_{\End}$ is finite (Proposition \ref{proposition: galois group endomorphism field is dihedral if real place}).
    We conclude that $\ker(\pi_2) \cap \text{image}(\rho_I^{\circ}) = \{1\}$. 
\end{proof}


\begin{remark}
    We can also define $\ell$-adic versions of the enhanced Galois representation: for every prime $\ell$ this is a group homomorphism $\Gal_F\rightarrow \Aut(\calO)\ltimes (\calO\otimes\Z_{\ell})^{\times}$ 
    encoding both the $\Gal_F$-action on $\calO$ and on the $\ell$-adic Tate module of $A$.
\end{remark}

\section{PQM surfaces over local and finite fields}\label{section:PQM surfaces over localfinite fields}

We collect some results about $\mathrm{PQM}$ surfaces $A$ over local and finite fields, especially the possible reduction types. 
The most important facts for our purposes are: a $\mathrm{PQM}$ surface $A/\Q$ of $\GL_2$-type has totally additive reduction at every bad prime (Corollary \ref{corollary:gl2-purelyadditive}); the prime-to-$p$ torsion in the totally additive case is controlled by the N\'eron component group (Lemma \ref{lem:prime-to-p-additive}); and the latter in turn is controlled by the smallest field extension over which $A$ acquires good reduction (Proposition \ref{prop: component group killed by good reduction field}).

For the remainder of this section, let $R$ be a henselian discrete valuation ring with fraction field $F$ of characteristic zero and perfect residue field $k$ of characteristic $p\geq 0$.

\subsection{N\'eron models of PQM surfaces}\label{subsec: neron models of PQM surfaces}

We first recall some notions in the theory of N\'eron models.
Let $A/F$ be an abelian variety with N\'eron model $\calA/R$. 
The special fiber $\calA_k$ fits into an exact sequence
$$
0\rightarrow \calA_k^{\circ} \rightarrow \calA_k \rightarrow \Phi\rightarrow 0
$$
where $\Phi$ is the \defi{component group} of $\calA_k$, a finite \'etale $k$-group scheme.  The identity component $\calA_k^0$ fits into an exact sequence
\begin{align}\label{equation: ses identity component neron model}
    0\rightarrow U\times T \rightarrow \calA^0_k\rightarrow B\rightarrow 0
\end{align}
where $U$ is a unipotent group, $T$ is a torus and $B$ is an abelian variety over $k$.
The dimensions of $U,T$, and $B$, which we denote by $u,t$, and $b$, are called the \defi{unipotent, toric and abelian ranks} of $A$, respectively. We have $ u + t + b = \dim A$, and $A$ has bad reduction if and only if $b < \dim A$. Similarly, $A$ has potentially good reduction over $F$ if and only if its toric rank is 0 over every finite extension of $F$. 

\begin{lemma}\label{lemma:PQMreduction}
Suppose that $A/F$ is an abelian surface such that $\End^0(A_{\Fbar})$ contains a non-split quaternion algebra. 
Then there exists a finite extension $F'/F$ such that $A_{F'}$ has good reduction. 
If $k$ is finite, we may take $F'$ to be a totally ramified extension of $F$.
\end{lemma}
\begin{proof}
 The fact that $A$ has potentially good reduction is well known, see e.g. \cite[p. 536]{CX08}. It follows from the fact that a non-split quaternion algebra does not embed in $\mathrm{Mat}_2(\Q)$, and hence does not embed in $\End(T)\otimes \Q$ for any torus $T/k$ of dimension $1$ or $2$.
 
 The last sentence of the lemma can be justified by taking a lift in $\Gal_{F}$ of the Frobenius in $\Gal_{k}$, in a manner analogous to \cite[p. 498]{SerreTate-GoodReduction}.
\end{proof}

\begin{prop}\label{prop:GL2totadditive}
Suppose that $A/F$ is an abelian surface such that $\End^0(A_{\Fbar})$ contains a non-split quaternion algebra. Suppose that $A$ has bad reduction. Then:
\begin{enumalph}
    \item $t = 0$.
    \item If $\End^0(A)$ contains a real quadratic field, then $u = \dim A = 2$. 
    \item If $u = 1$, then $A_K$ has good reduction over any field extension $K/F$ such that $\End^0(A_K)$ contains a real quadratic field.
\end{enumalph}
\end{prop}

\begin{proof}
$(a)$ follows from the fact that $A$ has potentially good reduction and the fact that the toric rank cannot decrease under extension of the base field \cite[Proposition 2.4]{CX08}. For $(b)$, we only need to exclude the possibility that $u=b=1$, so suppose by contradiction that it holds. Let $E\subset \End^0(A)$ be a real quadratic subfield. Reducing endomorphisms in \eqref{equation: ses identity component neron model} gives a (nonzero, hence injective) map $E \hookrightarrow \End^0(B)$. By assumption, $B$ is an elliptic curve. However, this contradicts the fact that the endomorphism algebra of an elliptic curve (over any field) does not contain a real quadratic field. Finally, $(c)$ follows from $(b)$, since the abelian rank cannot decrease after base change \cite[Proposition 2.4]{CX08}. \end{proof}

When $u = \dim A$ one says that $A$ has \defi{totally additive reduction}.

\begin{corollary}\label{corollary:gl2-purelyadditive}
Let $A/\Q$ be a $\mathrm{PQM}$ surface and $p$ a prime of bad reduction. 
Suppose that $A$ is of $\GL_2$-type.
Then $A$ has totally additive reduction at $p$.
\end{corollary}
\begin{proof}
    This follows from Proposition \ref{prop:GL2totadditive}(b) and the fact that $\End(A)$ is real quadratic by Proposition \ref{proposition: endomorphism rings of absolutely simple abelian surfaces over R}.
\end{proof}

\begin{remark}
One can show that if $p\geq 5$ then the Prym variety of $y^3 = x^4+x^2+p$ (which has $\mathrm{PQM}$ by \cite{LagaShnidman}) has unipotent rank $1$ over $\Q_p$. So the $\GL_2$-type hypothesis cannot be dropped in general in Corollary \ref{corollary:gl2-purelyadditive}.
\end{remark}

Finally, we state Raynaud's criterion for $A/F$ to have semistable reduction, which in the case of a $\mathrm{PQM}$ surface is necessarily good by Proposition \ref{prop:GL2totadditive}.

\begin{lemma}\label{RaynaudsCriterion}
    Let $A/F$ be a $\mathrm{PQM}$ surface, $n$ an integer not divisible by the residue characteristic $p$ and suppose that all points in $A[n]$ are defined over an unramified extension of $F$. Then
    \begin{enumalph}
        \item if $n=2$ then $A$ has good reduction over every ramified quadratic extension of $F$;
        \item if $n \geq 3$ then $A$ has good reduction over $F$.
    \end{enumalph}
\end{lemma}

\begin{proof}
    See \cite[\S7]{SilberbergZarhin95}.
\end{proof}

\subsection{The good reduction field and component group of a PQM surface}

Let $A/F$ be an abelian variety with potentially good reduction.
If $k$ is algebraically closed, there exists a smallest field extension $M/F$ such that $A_M$ has good reduction, called the \defi{good reduction field} of $A$.
This is a Galois extension, equal to $F(A[N])$ for every $N\geq 3$ coprime to $p$ \cite[\S2, Corollary 3]{SerreTate-GoodReduction}.
It is relevant for us because it controls the size of the component group, by the following result \cite[Theorem 1]{EdixhovenLiuLorenzini-ppartcomponentgroup}.

\begin{prop}\label{prop: component group killed by good reduction field}
    Suppose that $k$ is algebraically closed.
    Let $A/F$ be an abelian variety with potentially good reduction and reduction field $M/F$.
    Then the N\'eron component group $\Phi$ is killed by $[M:F]$.
\end{prop}

The next lemma constrains the good reduction field of a $\mathrm{PQM}$ surface.

\begin{lemma}\label{lemma: extension good reduction prime to 6}
Suppose that $k$ is algebraically closed.
Let $A/F$ be a $\mathrm{PQM}$ surface with good reduction field $M/F$.
Then $[M:F]$ divides $24^2$.
In particular, $[M:F]$ is coprime to any prime $\ell > 3$. 
\end{lemma}
\begin{proof}
Let $L$ be the endomorphism field of $A/F$ (Section \ref{subsection: endomorphism field}). 
By the N\'eron--Ogg--Shafarevich criterion, all prime-to-$p$ torsion is defined over $M$, hence $L\subset M$ by a result of Silverberg (Proposition \ref{prop: Silverberg result endo field}).  By Proposition \ref{proposition: galois group endomorphism field is dihedral if real place}, $[L:F]$ divides $24$.  By \cite[Proposition 4.2]{JordanMorrison} and its proof (whose notation does not agree with ours), we have $[M:L] \mid 24$. 
We conclude that $[M:F] = [M:L][L:F]$ divides $24^2$.
\end{proof}

\begin{lemma}\label{lemma: component group prime to p for p at least 5}
Let $A/F$ be a $\mathrm{PQM}$ surface and let $\ell \geq 5$. Then the order of $\Phi$ is not divisible by $\ell$.
\end{lemma}
\begin{proof}
Since formation of N\'eron models commutes with unramified base change, it is enough to prove the lemma in the case where $F$ has algebraically closed residue field. 
This then follows from Proposition \ref{prop: component group killed by good reduction field} and Lemma \ref{lemma: extension good reduction prime to 6}.
\end{proof}

We record the following technical lemma that will allow us to sometimes `quadratic twist away' bad primes.
This will be useful in the proof of Proposition \ref{prop: no (Z/2)^4}.

\begin{lemma}\label{lemma: quadratic twist away tot add primes with quad good red field}
    Suppose that $p\neq 2$.
    Let $A/F$ be an abelian variety with totally additive reduction. 
    Suppose that $A_M$ has good reduction for some quadratic extension $M/F$. 
    Then the quadratic twist $A^M$ of $A$ by $M$ has good reduction.
\end{lemma}
\begin{proof}
    Let $I_F$ and $I_M$ denote the inertia group of $\Gal_F$ and $\Gal_M$ respectively.
    Fix a prime $\ell\neq p$.
    By the N\'eron--Ogg--Shafarevich criterion, the $I_F$-action on the $\ell$-adic Tate module $T_{\ell}A$ factors through a faithful $I_F/I_M$-action, so acts via an element $\sigma\in \GL(T_{\ell}A)$ of order $2$.
    Since $A$ has totally additive reduction, $(T_{\ell}A)^{I_F} = 0$ and so $\sigma=-1$.
    Let $\chi_M\colon \Gal_F\rightarrow \{\pm 1\}$ be the character corresponding to the extension $M/F$.
    Then $T_{\ell}( A^M) \simeq T_{\ell}A \otimes \chi_M$ as $\Gal_F$-modules.
    Therefore $I_F$ acts trivially on $T_{\ell} (A^M)$ and $A^M$ has good reduction.
\end{proof}

\subsection{Component groups and torsion}
The relevance of the component group is the following well-known fact, see for example \cite[Remark 1.3]{Lorenzini-groupofcomponentsneronmodel}.
If $G$ is an abelian group, write $G^{(p)}$ for its subgroup of elements of finite order prime to $p$.

\begin{lemma}\label{lem:prime-to-p-additive}
If $A/F$ is an abelian variety with totally additive reduction $($i.e. $u = \dim A)$, then $A(F)^{(p)}_{\mathrm{tors}}$ is isomorphic to a subgroup of $\Phi(k)^{(p)}$, where $\Phi$ denotes the component group of $\mathcal{A}_k$.
\end{lemma}

Lorenzini has studied the component groups of general abelian surfaces with potentially good reduction and totally additive reduction, which leads to the following severe constraint on their torsion subgroups \cite[Corollary 3.25]{Lorenzini-groupofcomponentsneronmodel}.

\begin{theorem}[Lorenzini]\label{thm:lorenzini}
Let $A/F$ be an abelian surface with totally additive and potentially good reduction. Then $A(F)^{(p)}_{\mathrm{tors}}$ is a subgroup of one of the following groups:
\[\Z/5\Z, \, (\Z/3\Z)^2, \, (\Z/2\Z)^4, \, \Z/2\Z \times \Z/4\Z, \, \Z/2\Z \times \Z/6\Z.\]
\end{theorem}

We can say more if $A$ has totally additive reduction over any proper subextension of the good reduction field.
The following slight variant of \cite[Corollary 3.24]{Lorenzini-groupofcomponentsneronmodel} will be very useful in classifying torsion in the $\GL_2$-type case.

\begin{prop}\label{prop:gl2 type lorenzini variant}
Suppose that the residue field of $F$ is algebraically closed.
Let $A/F$ be an abelian variety with bad and potentially good reduction.
Let $M/F$ be the good reduction field of $A$.
Suppose that $A_{F'}$ has totally additive reduction for every $F\subset F' \subsetneq M$.
Suppose that the prime-to-$p$ torsion subgroup $A(F)^{(p)}_{\mathrm{tors}}$ of $A(F)$ is nontrivial.
Then there exists a prime number $\ell\neq p$ such that $[M:F]$ is a power of $\ell$ and $A(F)^{(p)}_{\mathrm{tors}} \simeq (\Z/\ell\Z)^k$ for some $k\geq 1$.
\end{prop}

\begin{proof}
Let $G := \Gal(M/F)$.
For every $F \subset F' \subsetneq M$, $A(F)^{(p)}_{\mathrm{tors}}\subset A(F')^{(p)}_{\mathrm{tors}}$ is isomorphic to a subgroup of the component group of $A_{F'}$ by Lemma \ref{lem:prime-to-p-additive}, which is killed by $[F:F']$ by Proposition \ref{prop: component group killed by good reduction field}.
By Galois theory, $A(F)^{(p)}_{\mathrm{tors}}$ is therefore killed by $\#H$ for every nontrivial subgroup $H \leq G$.
The group $A(F)^{(p)}_{\mathrm{tors}}$ is nontrivial by assumption; let $\ell$ be a prime dividing its order.
We claim that this $\ell$ satisfies the conclusions of the proposition.
Indeed, by definition of $A(F)^{(p)}_{\mathrm{tors}}$ we have $\ell \neq p$.
Moreover if $\#G$ is divisible by another prime $\ell'$, then by taking $H$ a Sylow-$\ell'$ subgroup of $G$ we get a contradiction, so $\#G = [M:F]$ is a power of $\ell$.
By taking $H$ to be an order $\ell$ subgroup of $G$, we see that $A(F)^{(p)}_{\mathrm{tors}}$ is killed by $\ell$, as desired.
\end{proof}

In the general case (not necessarily totally additive reduction), we have the following well-known result when $F$ is a finite extension of $\Q_p$, which follows from formal group law considerations \cite[\S2.5 and Proposition 3.1]{CX08}.

\begin{lemma}\label{lem:reduction is injective}
Suppose that $F/\Q_p$ is a finite extension of ramification degree $e$.
Let $A/F$ be an abelian variety with N\'eron model $\calA/R$.
Let $\red\colon A(F) =\calA(R)\rightarrow \calA(k)$ be the reduction map.
\begin{enumalph}
\item The restriction of $\red$ to prime-to-$p$ part of $A(F)_{\tors}$ is injective.
\item If in addition $e<p-1$, then $\red$ is injective on $A(F)_{\tors}$.
\end{enumalph}
\end{lemma}

\subsection{The conductor of a PQM surface}

Recall that the \defi{conductor} $\mathfrak{f}(A)$ of an abelian variety $A/\Q$ is a positive integer divisible exactly by the primes of bad reduction of $A$; see \cite{BrumerKramer-conductorabelianvariety} for a precise definition and more information.
We may write $\mathfrak{f}(A) = \prod_p p^{\mathfrak{f}_p(A)}$, where $\mathfrak{f}_p(A)$ denotes the \defi{conductor exponent} at a prime $p$.

\begin{lemma}\label{lemma: PQM surface at tame prime conductor exp 4}
    Let $A/\Q$ be a $\mathrm{PQM}$ surface of $\GL_2$-type.
    Let $p$ be a prime such that $A$ has bad reduction at $p$ but acquires good reduction over a tame extension of $\Q_p$.
    Then $\mathfrak{f}_p(A) = 4$.
\end{lemma}
\begin{proof}
    In that case $\mathfrak{f}_p(A)$ equals the tame conductor exponent at $p$, which is $2\times (\text{unipotent rank})+(\text{toric rank})$.
    This equals $2\times 2+0 =4$ by Proposition \ref{prop:GL2totadditive}.
\end{proof}

\begin{prop}\label{prop: conductor PQM surface of GL2 type}
Let $A/\Q$ be a $\mathrm{PQM}$ surface of $\GL_2$-type. 
Then the conductor of $A$ is of the form $2^{2i}3^{2j}N^4$, where $0\leq i\leq 10$, $0\leq j \leq 5$, and $N$ is squarefree and coprime to $6$.
\end{prop}
\begin{proof}
By Lemmas \ref{lemma: extension good reduction prime to 6} and \ref{lemma: PQM surface at tame prime conductor exp 4}, $\mathfrak{f}_p(A)=4$ for every bad prime $p\geq 5$.
The bounds $\mathfrak{f}_2(A)\leq 20$ and $\mathfrak{f}_3(A)\leq 10$ follow from a general result of Brumer--Kramer \cite[Theorem 6.2]{BrumerKramer-conductorabelianvariety}.
The fact that $\mathfrak{f}_2(A)$ and $\mathfrak{f}_3(A)$ are even follows from the fact that $\End^0(A)$ is a real quadratic field (Proposition \ref{proposition: endomorphism rings of absolutely simple abelian surfaces over R}) and \cite[(4.7.2)]{Serre-surlesrepsDuke}.
\end{proof}

\subsection{Finite fields}

Let $k = \F_q$ be a finite field of order $p^r$.
We will use the following two statements, whose proof can be found in \cite[\S2]{Jordan86}.

\begin{lemma}\label{lemma: Lpoly of QM surface is square}
Let $A/k$ be an abelian surface such that $\End^0(A)$ contains the quaternion algebra $B$. 
Then the characteristic polynomial of Frobenius is of the form $(T^2+aT+q)^2$ for some integer $a\in \Z$ satisfying $|a|\leq 2\sqrt{q}$.
\end{lemma}

\begin{prop}\label{prop: QM splits mod p}
Let $A/k$ be an abelian surface such that $\End^0(A)$ contains the quaternion algebra $B$. 
If $r$ is odd or $p \nmid \disc(B)$, then $A$ is isogenous to the square of an elliptic curve over $k$.
If $r$ is even and $p\mid \disc(B)$, $A_{\bar{k}}$ is isogenous to the square of a supersingular elliptic curve over $\bar{k}$.
\end{prop}

\section{Proof of Theorem \ref{thm:gl2type classification}: PQM surfaces of \texorpdfstring{$\GL_2$}{GL2}-type}\label{section: torsion gl2type}

Before proving Theorems \ref{thm:QMmazur ptorsion}-\ref{thm:QMmazur_groups}, it is useful to first prove Theorem \ref{thm:gl2type classification}, which classifies torsion subgroups of $\mathcal{O}$-PQM abelian surfaces $A$ over $\Q$ which are of $\GL_2$-type.  At a certain point in the argument we make use of the modularity of abelian surfaces of $\GL_2$-type, which we recall in \S\ref{subsec:modularity} and classify $\mathrm{PQM}$ surfaces of $\GL_2$-type with good reduction outside $2$ or $3$.
  In \S \ref{subsec:general constraints}, we deduce that a general $\calO$-$\mathrm{PQM}$ surface cannot have a full level $2$-structure over $\Q$.   In \S \ref{subsec:gl2type classification}, we prove Theorem \ref{thm:gl2type classification}.

\subsection{Abelian surfaces of \texorpdfstring{$\GL_2$}{GL2}-type and modular forms}\label{subsec:modularity}

\begin{theorem}\label{theorem: modularity theorem}
Let $A$ be an abelian surface such that $\End^0(A)$ is a real quadratic field.
Then the conductor of $A$ is of the form $N^2$ for some positive integer $N$, and there exists a unique Galois orbit $[f_A] \subset S_2(\Gamma_0(N))$ having coefficient field $K \simeq \End^0(A)$ whose local $L$-factors agree for each prime $p$:
\begin{equation} \label{eqn:Lpagree}
    L_p(A,T) = \prod_{\tau\colon K\hookrightarrow \C} L_p(\tau(f_A),T) \in 1+T\Z[T].
\end{equation}
Moreover, we have $[f_A] = [f_{A'}]$ if and only if $A$ is isogenous to $A'$ (over $\Q$).
\end{theorem}

\begin{proof}
As explained by Ribet \cite[Theorem (4.4)]{RibetGL2}, the fact that $A$ is of $\GL_2$-type over $\Q$ implies that $A$ is modular assuming Serre's modularity conjecture \cite[\S4.7, Theorem 5]{Serre-surlesrepsDuke}, which was proven by Khare--Wintenberger \cite{KhareWintenberger-Serresmodularity}.  Thus the equality of $L$-series \eqref{eqn:Lpagree} holds for some newform $f_A$.  Since $\End^0(A)$ is real, the character of $f_A$ is trivial \cite[Lemma (4.5.1)]{RibetGaloisAction}.  It follows from a theorem of Carayol \cite[Theoreme (A)]{Carayol-surlesrepsladiques} (local-global compatibility) that $A$ has conductor equal to $N^2$, where $N$ is the level of $f_A$.  
Finally, the fact that the Galois orbit of $f_A$ characterizes $A$ up to isogeny follows from the theorem of Faltings. 
\end{proof}

Recall that if $f\in S_2(\Gamma_0(N))$ is a newform and $\psi$ a primitive Dirichlet character, there exists a unique newform $g = f\otimes \psi$, the \defi{twist} of $f$ by $\psi$, whose $q$-expansion satisfies $a_n(g) = a_n(f) \psi(n)$ for all $n$ coprime to $N$ and the conductor of $\psi$.
If $f =g$, then $g$ is called a \defi{self-twist}. 
If $f$ and $g$ are Galois conjugate, $g$ is called an \defi{inner twist}.

\begin{prop}\label{prop: modularity theorem PQM surfaces}
Let $A$ be an abelian surface over $\Q$ such that $\End^0(A) \simeq \Q(\sqrt{m})$ with $m\geq 2$.  Then $A$ has $\mathrm{PQM}$ if and only if all of the following conditions hold:
\begin{enumroman}
    \item $f_A$ has no self-twists, equivalently $f_A$ is not CM;
    \item $f_A$ has a nontrivial inner twist by a Dirichlet character associated to a quadratic field $\Q(\sqrt{d})$; and
    \item The quaternion algebra $B_{d,m} \colonequals \quat{d,m}{\Q}$ is a division algebra.
\end{enumroman}
If all conditions \textup{(i)}--\textup{(iii)} hold, then in fact $\End^0(A_{\Qbar}) \simeq B_{d,m}$.
\end{prop}

\begin{proof}
    See Cremona \cite[\S2]{Cremona-abelianvarietieswithextratwist}.
\end{proof}
This reduces the enumeration of isogeny classes of $\GL_2$-type $\mathrm{PQM}$ surfaces $A$ over $\Q$ with fixed conductor to a computation in a space of modular forms.  

\begin{corollary}\label{cor: no gl2type good red outside 2}
There are no $\mathrm{PQM}$ surfaces $A$ over $\Q$ of $\GL_2$-type with good reduction outside $\{2\}$.
\end{corollary}

\begin{proof}
By Proposition \ref{prop: conductor PQM surface of GL2 type}, it is enough to check that there is no eigenform corresponding to a $\mathrm{PQM}$ surface of level $2^k$ for any $k \leq 10$. This information is contained in the LMFDB \cite{lmfdb} or \cite[Table 1]{GG09}. 
\end{proof}

\begin{corollary}\label{cor: gl2type good red outside 3}
There is exactly one isogeny class of $\mathrm{PQM}$ surfaces $A$ over $\Q$ of $\GL_2$-type with good reduction outside $\{3\}$: it has conductor $3^{10}$, any abelian surface $A$ in the isogeny class satisfies $A(\Q)_{\mathrm{tors}} \leq \Z/3\Z$.  
\end{corollary}

\begin{proof}
The fact that there is exactly one such isogeny class again follows from Proposition \ref{prop: conductor PQM surface of GL2 type} and information in the LMFDB or \cite[Table 1]{GG09}.
The corresponding Galois orbit of weight two newforms has LMFDB label \href{https://www.lmfdb.org/ModularForm/GL2/Q/holomorphic/243/2/a/d/}{\textsf{243.2.a.d}}.  From $L_2(1)=3$ and $L_{13}(1)=225$ we conclude that $\#A(\Q)_{\mathrm{tors}} \mid 3$ for every $A$ in this isogeny class.  
(In fact, the corresponding optimal quotient of $J_0(243)$ has $\Z/3\Z$ torsion subgroup by considering the image of the cuspidal subgroup of $J_0(243)$.)
\end{proof}

\begin{remark}
The isogeny class of Corollary \ref{cor: gl2type good red outside 3} has minimal conductor among all $\mathrm{PQM}$ surfaces $A$ of $\GL_2$-type.  It would be interesting to produce an explicit model over $\Q$; see also \cite[Question 2]{LagaShnidman}.
\end{remark}

\begin{prop}
There are exactly $44$ isogeny classes of $\mathrm{PQM}$ surfaces over $\Q$ of $\GL_2$-type with good reduction outside $\{2,3\}$.
\end{prop}

\begin{proof}
Again we use Propositions 
\ref{prop: conductor PQM surface of GL2 type} and \ref{prop: conductor PQM surface of GL2 type} to reduce the question to computing the number of Galois orbits of newforms in $S_2(\Gamma_0(N))$, where $N \mid 2^{10} 3^5$, with quadratic Hecke coefficient field, having an inner twist but no self-twist.
However, here we need to do a new computation in a large dimensional space.  
The code is available at \url{https://github.com/ciaran-schembri/QM-Mazur}; we provide a few details to explain how we proceeded, referring to the book by Stein \cite{Stein} on modular symbols and more broadly \cite{allCMF} for a survey of methods to compute modular forms.

We work with modular symbols, and we loop over all possible (imaginary) quadratic characters $\psi$ supported at $2,3$, corresponding to inner twist.  For each character $\psi$, of conductor $d$: 
\begin{itemize}
\item For a list of split primes $p \geq 5$, we inductively compute the kernels of $T_p-a$ where $|a| \leq 2\sqrt{p}$.  
\item For a list of inert primes $p \geq 5$, we further inductively compute the kernels of $T_p^2-db^2$ where $db^2 \leq 4p$.
\end{itemize}
The first bound holds since $\psi(p)=1$ so $a_p(f)\psi(p)=\tau(a_p(f)) = a_p(f)$ so $a_p(f) \in \Z$, and the Ramanujan--Petersson bound holds; the second bound holds since $\psi(p)=-1$ now gives $\tau(a_p(f))=-a_p(f)$ so $a_p(f)=\sqrt{d} b$ with again $\sqrt{d}|b| \leq 2\sqrt{p}$.  It is essential to compute the split primes first, and only compute the induced action of $T_p$ on the kernels computed in the first step.  

To simplify the linear algebra, we work modulo a large prime number $q$, checking that each Hecke matrix $T_p$ (having entries in $\Q$) has no denominator divisible by $q$.  The corresponding decomposition gives us an `upper bound': if we had the desired eigenspace for $T_p$, it reduces modulo $q$, but a priori some of these spaces could accidentally coincide or the dimension could go down (corresponding to a prime of norm $q$ in the Hecke field).  To certify the `lower bound', we compute a small linear combination of Hecke operators supported at split primes and use the computed eigenvalues to recompute the kernel over $\Q$ working with divisors $N' \mid N$, and when we find it we compute the dimension of the oldspace for the form at level $N'$ inside level $N$ and confirm that it matches the dimension computed modulo $q$.  

In fact, we find that $N \mid 2^8 3^5$ or $N \mid 2^{10} 3^4$.  (Indeed, a careful analysis of the possible endomorphism algebra can be used to show this a priori.)  

To certify that the form is not PCM, we find a coefficient for an inert prime that is nonzero.  That the form has the correct inner twist by $\psi$ is immediate: the form would again appear somewhere in our list, so once we have identified the newforms uniquely with coefficients, the inner twist must match, Sherlock Holmes-style.  We similarly discard the forms with PCM.  

Finally, we compute the split $\mathrm{PQM}$ forms by identifying the quaternion algebra above using Proposition \ref{prop: modularity theorem PQM surfaces}.
\end{proof}

The complete data is available online (\url{https://github.com/ciaran-schembri/QM-Mazur}); we give a summary in Table \ref{table:mfs}, listing forms in a fixed level, up to (quadratic) twist.

For example, Table \ref{table:mfs} says that up to twist there are $3$ newforms of level $N=20736=2^8 3^4$, each having $4$ Galois newform orbits for a total of $12$ newform orbits.

\begin{center}
 \begin{longtable}{p{3cm}<{\centering} | p{0.7cm}<{\centering} | p{1.4cm}<{\centering} | p{0.7cm}<{\centering}| p{6cm}<{\centering}}
  \caption{Twist classes of modular forms corresponding to PQM abelian surfaces over $\Q$ of $\GL_2$-type with good reduction outside $\{2,3\}$}\\
  $N$ & $\psi$ & $\disc B$ & num & LMFDB labels \\
  \cline{1-5}
  $243=3^5$ & $-3$ & $6$ & $1$ & \lmfdbmf{243.2.a.d} \\
  $972=2^2 3^5$ & $-3$ & $6$ & $1$ & \lmfdbmf{972.2.a.e} \\
  $2592=2^5 3^4$ & $-4$ & $6$ & $2$ & \lmfdbmf{2592.2.a.l}, \lmfdbmf{2592.2.a.p} \\
  $2592=2^5 3^4$ & $-4$ & $6$ & $2$ & \lmfdbmf{2592.2.a.m}, \lmfdbmf{2592.2.a.r} \\
  $3888=2^4 3^5$ & $-3$ & $6$ & $2$ & \lmfdbmf{3888.2.a.b}, \lmfdbmf{3888.2.a.t} \\
  $5184=2^6 3^4$ & $-4$ & $6$ & $2$ & \lmfdbmf{5184.2.a.bl}, \lmfdbmf{5184.2.a.bx} \\
  $5184=2^6 3^4$ & $-4$ & $6$ & $2$ & \lmfdbmf{5184.2.a.bk}, \lmfdbmf{5184.2.a.bv} \\
  $15552=2^6 3^5$ & $-3$ & $6$ & $2$ \\
  $15552=2^6 3^5$ & $-3$ & $6$ & $2$ \\
  $20736=2^8 3^4$ & $-4$ & $6$ & $4$ \\
  $20736=2^8 3^4$ & $-4$ & $22$ & $4$ \\
  $20736=2^8 3^4$ & $-8$ & $10$ & $4$ \\
  $62208=2^8 3^5$ & $-3$ & $6$ & $4$ \\
  $62208=2^8 3^5$ & $-3$ & $6$ & $4$ \\
  $82944=2^{10} 3^4$ & $-24$ & $6$ & $4$ \\
  $82944=2^{10} 3^4$ & $-24$ & $6$ & $4$ &
 \label{table:mfs}
 \end{longtable}
 \end{center}

\begin{corollary} \label{cor:only9}
If $A$ is a $\mathrm{PQM}$ abelian surface of $\GL_2$-type over $\Q$ with good reduction outside $\{2,3\}$ and $\#A(\Q)_{\textup{tors}}$ nontrivial, then $A$ corresponds to either \textup{\lmfdbmf{243.2.a.d}} or \textup{\lmfdbmf{972.2.a.e}}.  In particular, $\#A(\Q)_{\textup{tors}} \leq 9$.
\end{corollary}

\begin{proof}
Direct calculation as in Corollary \ref{cor: gl2type good red outside 3}.
\end{proof}

\subsection{Full level \texorpdfstring{$2$}{2}-structure}\label{subsec:general constraints}

Before imposing the $\GL_2$-type assumption in the next subsection, we show that $\calO$-$\mathrm{PQM}$ surfaces cannot have full level $2$-structure over $\Q$. 

\begin{prop}\label{prop: no (Z/2)^4}
Let $A/\Q$ be an $\calO$-$\mathrm{PQM}$ surface. Then $A(\Q)[2] \not\simeq (\Z/2\Z)^4$.
\end{prop}
\begin{proof}
Suppose $A(\Q)[2] \simeq (\Z/2\Z)^4$.
Since $A[2]$ is free of rank one as an $\calO/2\calO$-module and contains a $\Q$-rational generator, we have $A[2] \simeq \calO/2\calO$ as $\Gal_\Q$-modules. By Theorem \ref{theorem: fixed points of O/N} and Proposition \ref{proposition: galois group endomorphism field is dihedral if real place}, this implies that the endomorphism field $L/\Q$ is quadratic, so that $A$ has $\GL_2$-type by Lemma \ref{lemma: A GL2-type iff endo field quadratic}.

Let $K$ be a quadratic field ramified at all primes $p\geq 3$ of bad reduction of $A$ and unramified at all primes $p\geq 3$ of good reduction.
Corollary \ref{corollary:gl2-purelyadditive} and Lemmas \ref{RaynaudsCriterion}(a), \ref{lemma: quadratic twist away tot add primes with quad good red field} and \ref{lemma: quadratic twisting doesnt change endo field} show that the quadratic twist of $A$ by $K$ is an $\calO$-$\mathrm{PQM}$ surface of $\GL_2$-type with good reduction outside $\{2\}$.
But by Corollary \ref{cor: no gl2type good red outside 2}, no such surface exists.
\end{proof}

\subsection{Torsion classification in the \texorpdfstring{$\GL_2$}{GL2}-type case}\label{subsec:gl2type classification}
Now we assume $A/\Q$ is a $\mathrm{PQM}$ surface of $\GL_2$-type.  By Lemma \ref{lemma: A GL2-type iff endo field quadratic}, there exists a quadratic extension $L/\Q$ (the endomorphism field) such that $\End(A_L) = \End(A_{\Qbar})$.

\begin{lemma}\label{lem: honda-tate}
If $\ell$ is a prime such that $A[\ell](\Q) \neq 0$, then $\ell \leq 7$.
\end{lemma}
\begin{proof}
By Lemma \ref{lemma:PQMreduction}, there exists a finite extension $L'/L$ that is totally ramified at $2$ and such that $A_{L'}$ has good reduction.
Let $\mathfrak{q}$ be a prime in $L'$ above $2$ and let $k$ be its residue field.
Since $L/\Q$ is quadratic, $k$ is isomorphic to $\F_2$ or $\F_4$.
Therefore the reduction of $A_{L'}$ at $\mathfrak{q}$ is an abelian surface $B$ over $k$ such that $\End^0(B)$ contains $\End^0(A_L)$.
By Lemma \ref{lem:reduction is injective}, $B[\ell](k)\neq 0$ and so $\ell$ divides $\#B(\F_4)$.
On the other hand, Lemma \ref{lemma: Lpoly of QM surface is square} shows that the $L$-polynomial of $B_{\F_4}$ is of the form $(T^2+aT+4)^2$ with $a\in \Z$ satisfying $|a|\leq 2\sqrt{4} = 4$.
Therefore $\ell$ divides $(1+a+4)^2$, hence $\ell$ divides $(1+a+4)\leq 9$, hence $\ell\leq 9$.
\end{proof}

\begin{lemma}\label{lem:goodawayfromell}
If $\ell \geq 5$ is a prime such that $A[\ell](\Q) \neq 0$, then $A/\Q$ has good reduction away from $\ell$.
\end{lemma}
\begin{proof}
Let $p$ be a prime of bad reduction of $A$.
Since $A$ is of $\GL_2$-type, the algebra $\End^0(A)$ is a quadratic field; it is real quadratic by Proposition \ref{proposition: endomorphism rings of absolutely simple abelian surfaces over R}.
Proposition \ref{prop:GL2totadditive}(c) implies that $A$ has totally additive reduction at $p$.
By Lemmas \ref{lemma: component group prime to p for p at least 5} and \ref{lem:prime-to-p-additive}, we must have $p=\ell$.
We conclude that $A$ has good reduction outside $\{\ell\}$.
\end{proof}

\begin{prop}\label{prop: no 5 or 7}
 If $\ell$ is a prime such that $A[\ell](\Q) \neq 0$, then $\ell \in \{2,3\}$.    
\end{prop}
\begin{proof}
Suppose that $\ell \geq 5$.
By Proposition \ref{proposition: endomorphism rings of absolutely simple abelian surfaces over R}, the quadratic extension $L/\Q$ is imaginary quadratic. Moreover, by a result of Silverberg \cite[Theorem 4.2]{Silverberg92a}, the surface $A$ has bad reduction at all primes ramifying in $L$. By Lemma \ref{lem:goodawayfromell}, $L$ is therefore only ramified at $\ell$.  If $\ell = 5$, this is already a contradiction since there are no imaginary quadratic fields ramified only at $5$. If $\ell = 7$, then we conclude that $L = \Q(\sqrt{-7})$. Since $2$ splits in $L$, this means that the residue field in the proof of Lemma \ref{lem: honda-tate} is equal to $\F_2$. Continuing with the proof there, we deduce the stronger inequality $|a| \leq 2\sqrt{2}$, and we find that $\ell$ divides $1 + a + 2 < 6$, which is a contradiction.
\end{proof}

\begin{remark}
We can also deduce Proposition \ref{prop: no 5 or 7} from Lemma \ref{lem:goodawayfromell} by invoking modularity (Proposition \ref{prop: modularity theorem PQM surfaces}), the fact that such an abelian surface must have conductor $\ell^4$ (Proposition \ref{prop: conductor PQM surface of GL2 type}) and the fact that there are no $\mathrm{PQM}$ eigenforms in $S_2(\Gamma_0(25))$ or $S_2(\Gamma_0(49))$. We also note that Schoof has proven that there are no abelian varieties with everywhere good reduction over $\Q(\zeta_\ell)$ for various small $\ell$, including $5$ and $7$ \cite{schoof}.
\end{remark}

\begin{prop}\label{prop:gl2 constrains}
Either $A(\Q)_{\tors}\subset (\Z/2\Z)^3$ or $A(\Q)_{\tors}\subset (\Z/3\Z)^2$.
\end{prop}
\begin{proof}
By Proposition \ref{prop: no 5 or 7}, $A(\Q)_{\tors}$ is a group of order $2^i3^j$. 
We may assume that $A(\Q)_{\mathrm{tors}}\neq 0$; let $\ell \in \{2,3\}$ be such that $A[\ell](\Q) \neq 0$.

Suppose there exists a prime $p\geq 5$ of bad reduction. 
Then $A$ has totally additive reduction over every finite extension $F/\Q_p$ over which it has bad reduction by Proposition \ref{prop:GL2totadditive}.
Therefore the assumptions of Proposition \ref{prop:gl2 type lorenzini variant} apply for $F= \Q_p^{\mathrm{nr}}$ (the maximal unramified extension of $\Q_p$), and so $A(\Q)_{\mathrm{tors}} = A(\Q)_{\mathrm{tors}}^{(p)} \subset A(\Q_p^{\mathrm{nr}})_{\mathrm{tors}}^{(p)}\simeq (\Z/\ell\Z)^k$ for some $1 \leq k\leq 4$. If $\ell = 2$, then $k \leq 3$ by Proposition \ref{prop: no (Z/2)^4}. If $\ell = 3$, then $k \leq 2$, since $A(\Q)^{(2)}_{\mathrm{tors}} \hookrightarrow A_2(\F_2)$ for some abelian surface $A_2/\F_2$ (using Lemmas \ref{lemma:PQMreduction} and \ref{lem:reduction is injective}) and $\#A_2(\F_2) \leq 25$ for all such surfaces.
We conclude 
$A(\Q)_{\tors}\subset (\Z/2\Z)^3$ or $A(\Q)_{\tors}\subset (\Z/3\Z)^2$, as desired.

It remains to consider the case that $A$ has good reduction outside $\{2,3\}$.
A computation with modular forms of level dividing $2^{10}\cdot 3^5$ shows that $\# A(\Q)_{\tors} \mid 9$ for such surfaces by Corollary \ref{cor:only9}, but we give an argument that only involves computing modular forms of much smaller level.
We may assume $A$ has bad reduction at both of these primes by Corollaries \ref{cor: no gl2type good red outside 2} and \ref{cor: gl2type good red outside 3}. If $A[2](\Q) = 0$, then Proposition \ref{prop:gl2 type lorenzini variant} shows again that $A(\Q)_{\mathrm{tors}}=A(\Q)_{\mathrm{tors}}^{(2)} \subset A(\Q_2^{\mathrm{nr}})_{\mathrm{tors}}^{(2)}\subset (\Z/3\Z)^2$.
Similarly $A(\Q)_{\mathrm{tors}}\subset (\Z/2\Z)^3$ if $A[3](\Q)=0$.
Thus, it remains to rule out the possibility that $A(\Q)$ contains a point of order $6$. In that case, Proposition \ref{prop:gl2 type lorenzini variant} shows that the extensions $M_2/\Q^{\mathrm{nr}}_2$ and $M_3/\Q^{\mathrm{nr}}_3$ over which $A$ attains good reduction have degrees that are powers of $3$ and $2$ respectively, and hence are tamely ramified. Hence $A$ has conductor $2^43^4$ by Lemma \ref{lemma: PQM surface at tame prime conductor exp 4} and corresponds to an eigenform of level $2^23^2 = 36$, by Theorem \ref{theorem: modularity theorem}. However, there are no $\mathrm{PQM}$ eigenforms of level $36$ \cite[Table 1]{GG09}. 
\end{proof}

Next we constrain the torsion even further and show that $(\Z/2\Z)^3$ does not occur.  For this, we combine a cute fact from linear algebra with a purely local proposition that makes use of the enhanced Galois representation of \S\ref{enhanced representation}.

\begin{lemma}\label{lem:linalg}
Let $k$ be a field and $V\subset \calO_k := \calO \otimes_{\Z} k$ a $3$-dimensional $k$-subspace. Then $V$ contains an $\calO_k$-module generator of $\calO_k$.
\end{lemma}
\begin{proof}
If $\calO_k$ is a division algebra, every nonzero element of $V$ is an $\calO_k$-generator.
If the characteristic of $k$ divides $\disc(B)$, the lemma follows from Lemma \ref{lemma: classification submodules ramified case} and the fact that the ideal $J$ described there is $2$-dimensional.
It suffices to consider the case when $\calO_k\simeq \Mat_2(k)$ and to prove that in this case $V$ contains an invertible matrix. 
(This is well known, we give a quick proof here.)
Suppose otherwise.
If $k$ admits a quadratic field extension $k'$, then embedding $k' \subset \mathrm{Mat}_2(k)$, we compute $\dim(V + k') = \dim V + \dim k' - \dim(V \cap k') = 3 + 2 - 0 = 5$, which is a contradiction. In general, the subspace $V$ is defined over a subfield $k''$ of $k$ which is finitely generated over its prime field. The previous argument then applies over $k''$.
\end{proof}

Recall that $\Q_p^{\mathrm{nr}}$ denotes the maximal unramified extension of $\Q_p$.

\begin{prop}\label{prop: (Z/2)^3 implies good red quad extension}
    Let $p$ be an odd prime, $F$ a finite extension of $\Q_p^{\mathrm{nr}}$ and $A/F$ an $\calO-\mathrm{PQM}$ surface with $(\Z/2\Z)^3\subset A[2](F)$.
    Then $A$ acquires good reduction over every quadratic extension of $F$.
\end{prop}
\begin{proof}
    If $A[2](F)\simeq (\Z/2\Z)^4$, this immediately follows from Raynaud's criterion (Lemma \ref{RaynaudsCriterion}(a)), so assume that $A[2](F)\simeq (\Z/2\Z)^3$.
    By Lemma \ref{lem:linalg}, there exists an $F$-rational $\calO/2\calO$-generator $P\in A[2](F)$, and hence $A[2]\simeq \calO/2\calO$ as $\Gal_F$-modules. 
    
    Let $L/F$ be the endomorphism field of $A_F$ and let $M/F$ be the smallest field over which $A_F$ acquires good reduction.
    By the N\'eron-Ogg-Shafarevich criterion, $M = F(A[4])$.
    By Proposition \ref{prop: Silverberg result endo field}, $L\subset M$.
    Since $A[2]\simeq \calO/2\calO$ as $\Gal_{\Q}$-modules, $F(A[2]) \subset L$.
    We therefore have a chain of inclusions $F \subset F(A[2]) \subset L\subset M = F(A[4])$.
    Since $A[2](\Q) \simeq (\Z/2\Z)^3$, $F(A[2])/F$ is a $(2,2,\dots,2)$-extension.
    The same is true for $F(A[4])/F(A[2])$.
    Since $p$ is odd and the residue field is algebraically closed, both these extensions are cyclic, so at most quadratic.
    Therefore $F(A[2])/F$ is a quadratic extension. 
    If $L \neq F(A[2])$, then $L/F$ would be cyclic of order $4$, and there would be an order $4$ element $g\in \Aut(\calO)$ whose fixed points on $\calO/2\calO$ is $(\Z/2\Z)^3$.
    A calculation similar to the proof of the $D_4$ case in Theorem \ref{theorem: fixed points of O/N} shows that this is not possible. 
    We conclude that $L = F(A[2])$ and that $M / L$ is at most quadratic.
    
    To prove the proposition, it suffices to prove that $M/F$ is quadratic, so assume by contradiction that this is not the case.
    Then $M/L$ and $L/F$ are both quadratic and $\Gal(M/F)= \{1,g,g^2,g^3\}$ is cyclic of order $4$.
    
    Consider the mod $4$ Galois representation $\rho\colon \Gal_F\rightarrow \GL(A[4])$, which factors through $\Gal_F \rightarrow \Gal(M/F)$.
    Let $Q\in A[4](M)$ be a lift of the $\calO/2\calO$-generator $P\in A[2](F)$.
    Then $Q$ is an $\calO/4\calO$-generator for $A[4]$, and hence by the enhanced Galois representation construction, we know that $\rho \simeq \rho_4^\circ$ lands in $\Gal(L/F)  \ltimes (\calO/4\calO)^{\times} $ (see \S\ref{enhanced representation} and Proposition \ref{prop: enhanced gal rep iso to normal gal rep if N >2}).
    The situation can be summarized as follows:
    \[\begin{tikzcd}
	{\Gal(M/L)} & {(\calO/4\calO)^{\times}} \\
	{\Gal(M/F)} & { \Gal(L/F) \ltimes (\calO/4\calO)^{\times}} \\
	{\Gal(L/F)} & {\Gal(L/F) \ltimes  (\calO/2\calO)^{\times}}
	\arrow[two heads, from=2-1, to=3-1]
	\arrow["\rho_4^\circ", hook, from=2-1, to=2-2]
	\arrow["\rho_2^\circ",hook, from=3-1, to=3-2]
	\arrow["\rho_4^\circ|_{\Gal_L} ",hook, from=1-1, to=1-2]
	\arrow[hook, from=1-2, to=2-2]
	\arrow[two heads, from=2-2, to=3-2]
	\arrow[hook, from=1-1, to=2-1]
\end{tikzcd}\]
    The horizontal maps are the enhanced Galois representations for $L$ mod $4$, $F$ mod $4$ and $F$ mod $2$ respectively. 
    Write $\Gal(L/F) = \{1,\sigma\}$.
    Since $P$ is $F$-rational, the bottom map sends $\sigma$ to $(\sigma,1)$.
    By commutativity of the bottom square, $\rho_4^\circ(g) = (\sigma,x)$, where $x\in (\calO/4\calO)$ satisfies $x\equiv 1 \mod 2\calO$.
    Since $A_L$ has bad and hence totally additive reduction by Proposition \ref{prop:GL2totadditive}, the nontrivial element of $\Gal(M/L)$ maps to $-1$ in $(\calO/4\calO)^{\times}$. (In fact, the generator of $\Gal(M/L)$ even maps to $-1$ in $\GL(T_2A)$ by an argument identical to the proof of Lemma \ref{lemma: quadratic twist away tot add primes with quad good red field}.)
    By the commutativity of the top diagram, $(\sigma,x)^2 = (1,-1)$.
    The involution $\sigma$ acts on $(\calO/4\calO)^{\times}$ by conjugating by an element $b\in \calO \cap N_{B^{\times}}(\calO)$ whose fixed points on $\calO/2\calO$ are $(\Z/2 \Z)^3$.
    Therefore $(\sigma,x)^2 = (1,-1)$ is equivalent to $b^{-1}xbx = -1$.
    By Lemma \ref{lemma: mod 4 quaternion calculation (z/2)^3}, no such $x$ exists, obtaining the desired contradiction.
\end{proof}

\begin{prop}\label{prop: GL2 type no (Z/2)^3 torsion}
    Let $A/\Q$ be an $\calO$-$\mathrm{PQM}$ surface of $\GL_2$-type. Then $(\Z/2\Z)^3\not\subset A(\Q)[2]$.
\end{prop}
\begin{proof}
    Let $K$ be a quadratic field ramified at all primes $p\geq 3$ of bad reduction of $A$ and unramified at all primes $p\geq 3$ of good reduction.
    Corollary \ref{corollary:gl2-purelyadditive}, Proposition \ref{prop: (Z/2)^3 implies good red quad extension} and Lemmas \ref{lemma: quadratic twist away tot add primes with quad good red field} and \ref{lemma: quadratic twisting doesnt change endo field} show that the quadratic twist of $A$ by $K$ is an $\calO$-$\mathrm{PQM}$ surface of $\GL_2$-type with good reduction outside $\{2\}$.
    But no such $\calO$-$\mathrm{PQM}$ surface exists by Corollary \ref{cor: no gl2type good red outside 2}.
\end{proof}

We are finally ready to prove our classification result for torsion subgroups of $\calO$-$\mathrm{PQM}$ surfaces of $\GL_2$-type. 

\begin{proof}[Proof of Theorem $\ref{thm:gl2type classification}$]
 
By Propositions \ref{prop:gl2 constrains} and \ref{prop: GL2 type no (Z/2)^3 torsion}, we have ruled out all groups aside from those listed in the theorem. It remains to exhibit infinitely many abelian surfaces $A/\Q$ of $\GL_2$-type with torsion subgroups isomorphic to each of the groups 
\[\{0\}, \Z/2\Z, \Z/3\Z, (\Z/2\Z)^2,(\Z/3\Z)^2.\] 
Let $\calO_6$ be the maximal quaternion order of reduced discriminant $6$ (unique up to isomorphism). In \cite[\S9]{LagaShnidman}, one-parameter families of $\GL_2$-type $\calO_6$-PQM surfaces with generic torsion subgroups $\{0\}, \Z/2\Z$, $\Z/3\Z$ and $(\Z/3\Z)^2$ are given among Prym surfaces of bielliptic Picard curves. In Proposition \ref{prop:Z2Z2} below, we give a one-parameter family of $\GL_2$-type $\calO_6$-PQM Jacobians $J$ with $(\Z/2\Z)^2 \subset J(\Q)_{\tors}$.
\end{proof}

To state the next result, we define the rational functions
\begin{align*}
j(T) &=\frac{(-64T^{20} + 256T^{16} - 384T^{12} + 256T^8 - 64T^4)}{(T^{24} + 42T^{20} + 591T^{16} + 2828T^{12} + 591T^8 + 42T^4 + 1)}; \\
J_2(T) &=12(j+1); \\
J_4(T) &=6(j^2+j+1); \\
J_6(T) &=4(j^3-2j^2+1); \\
J_8(T) &=(J_2J_6-J_4^2)/4; \\
J_{10}(T) &=j^3.
\end{align*}

\begin{prop}\label{prop:Z2Z2}
For all but finitely many $t \in \Q$, there exists a genus two curve $C_t/\Q$ with Igusa invariants $(J_2(t): J_4(t) : J_6(t) : J_8(t) : J_{10}(t))$, whose Jacobian $J_t/\Q$ is an $\calO_6$-$\mathrm{PQM}$ surface of $\GL_2$-type and satisfies $J_t(\Q)_{\tors} \supset (\Z/2\Z)^2$.
\end{prop}

\begin{proof}
In \cite[p.742]{BG08}, the authors have an expression for Igusa-Clebsch invariants (which we have translated to Igusa invariants) of genus 2 curves defining $\calO$-$\mathrm{PQM}$ surfaces for every value of a parameter $j$ (which is a coordinate on the full Atkin-Lehner quotient of the discriminant 6 Shimura curve). The field of moduli for $k_{R_3}$, in their notation, is $\Q(\sqrt{-27 - 16j^{-1}})$ and the obstruction for these genus 2 curves to be defined over $\Q$ is given by the Mestre obstruction $\Big(\frac{-6j, -2(27j+16)}{\Q}\Big).$ A short computation for the family $j(T)$ shows that $-27-16j^{-1}$ is a square in $\Q(T)^\times$, and hence $k_{R_3} = \Q$ for all non-singular specializations. Furthermore,  one checks that the Mestre obstruction also vanishes for all such $t$. Thus, the Igusa invariants in the statement of the proposition give an infinite family of $\calO$-$\mathrm{PQM}$ Jacobians $J/\Q$ of $\GL_2$-type with $\End^0(J) \simeq \Q(\sqrt{3}).$
(Only finitely many $j\in \Q$ correspond to CM points \cite[\S5, Table 1]{BG08}, so $J$ is geometrically simple for all but finitely many $t\in \Q$.)

Using Magma, one can write down an explicit sextic polynomial $f_T(x)$ such that $C_t$ has model $y^2 = f_t(x)$. The coefficients of $f_T(x)$ are too large to include here, but we have posted them \href{https://github.com/ciaran-schembri/QM-Mazur}{here}. We find that there is a factoriztion
\[f_T(x) = q_{1,T}(x)q_{2,T}(x)q_{3,T}(x)\]  
where each $q_{i,T}$ is a quadratic polynomial in $\Q(T)[x]$. From this we see that for all but finitely many $t$, the group $(\Z/2\Z)^2$ is a subgroup of $J_t(\Q)_{\tors}.$ Indeed, $J_t = \Pic_0(C_t)$ and for each $i \in \{1,2,3\}$, the divisor class $(\alpha,0) - (\alpha',0)$, where $q_{i,t}(x) = (x-\alpha)(x - \alpha')$, is defined over $\Q$ and has order $2$.  In future work, we will explain how the special family $j(T)$ was found using the arithmetic of Shimura curves. 
\end{proof}

\section{Proof of Theorem \ref{thm:QMmazur ptorsion}: reduction to \texorpdfstring{$\GL_2$}{GL2}-type}

 In this section, we prove Theorem \ref{thm:QMmazur ptorsion}. By Theorem \ref{thm:gl2type classification}, it is enough to prove:

\begin{theorem}\label{theorem: reduction to GL2-type}
Let $A/\Q$ be an $\calO$-$\mathrm{PQM}$ surface,
and let $\ell \geq 5$ be a prime number such that $A[\ell](\Q)\neq 0$. 
Then $A$ is of $\GL_2$-type.
\end{theorem}

Theorem \ref{theorem: reduction to GL2-type} follows from combining Propositions \ref{proposition: exclude torsion ramified case} and \ref{proposition: large torsion unramified case means GL2-type} below. The proofs consist mostly of careful semi-linear algebra over non-commutative rings, combined with a small drop of global arithmetic input.

\subsection{Linear algebra}\label{subsection: linear algebra in calO mod p}

Let $\ell$ be a prime and $\calO_{\ell} := \calO\otimes \F_{\ell}$.
If $\ell\nmid \disc(B)$, then $\calO_{\ell} \simeq \Mat_2(\F_{\ell})$, since $\calO$ is maximal.
If $\ell \mid \disc(B)$, then $\calO_{\ell}$ is isomorphic to the nonsemisimple algebra \cite[\S4]{Jordan86}
\begin{align}\label{equation: description OB mod p for p ramified}
    \left\{\begin{pmatrix} \alpha & \beta \\ 0 & \alpha^{\ell}  \end{pmatrix} \mid \alpha,\beta\in \F_{{\ell}^2} \right\}\subset \Mat_2(\F_{\ell^2}).
\end{align}
In both cases, we will describe all left ideals of $\calO_{\ell}$.
Equivalently, given a left $\calO_{\ell}$-module $M$, free of rank one, we will describe all its (left) $\calO_{\ell}$-submodules.

First we suppose that $\ell\nmid \disc(B)$; fix an isomorphism $\calO_{\ell}\simeq \Mat_2(\F_{\ell})$ and a free rank one left $\calO_{\ell}$-module $M$.
Let $e_1,e_2,w$ be the elements of $\calO_{\ell}$ corresponding to the matrices 
\begin{align*}
    \begin{pmatrix} 1 & 0 \\ 0 & 0  \end{pmatrix}, \begin{pmatrix} 0 & 0 \\ 0 & 1  \end{pmatrix}, 
    \begin{pmatrix} 0 & 1 \\ 1 & 0  \end{pmatrix}
\end{align*}
respectively. 
Then $e_1$,$e_2$ are idempotents satisfying $e_1e_2=0$, $e_1+e_2=1$ and $e_1w = we_2$.
Set $M_i = \ker(e_i \colon M\rightarrow M)\subset M$.
Then $M = M_1\oplus M_2$ and $w$ induces mutually inverse bijections $M_1\rightarrow M_2$ and $M_2\rightarrow M_1$.
Given an $\calO_{\ell}$-submodule $N\subset M$, define $N_i := N\cap M_i$.
Since $N$ is $\calO_{\ell}$-stable, $N = N_1\oplus N_2$ and $w(N_1) = N_2$.
\begin{lemma}[Unramified case]\label{lemma: classification submodules of Mat2}
The map $N\mapsto (N_1,N_2)$ induces a bijection between left $\calO_{\ell}$-submodules of $M$ and pairs of $\F_{\ell}$-subspaces  $(N_1\subset M_1,N_2\subset M_2)$ satisfying $w(N_1) = N_2$.
\end{lemma}
\begin{proof}
This is elementary, using the fact that $\calO_{\ell}$ is generated (as a ring) by $e_1,e_2$ and $w$.
\end{proof}

Next suppose that $\ell$ divides $\disc(B)$ and fix an isomorphism between $\calO_{\ell}$ and the ring described in \eqref{equation: description OB mod p for p ramified}.
The set of strictly upper triangular matrices is a two-sided ideal $J\subset \calO_{\ell}$ that satisfies $\calO_{\ell}/J\simeq \F_{\ell^2}$. The following lemma is easily verified \cite[\S4]{Jordan86}.

\begin{lemma}[Ramified case]\label{lemma: classification submodules ramified case}
The only proper left ideal of $\calO_{\ell}$ is $J$.
Consequently, the only proper $\calO_{\ell}$-submodule of $M$ is $M[J] = \{m\in M\mid j\cdot m=0 \text{ for all }j\in J\}$.
\end{lemma}

\subsection{The subgroup generated by a torsion point}

Let $A/\Q$ be an $\calO$-$\mathrm{PQM}$ surface and $\ell$ be a prime number.
Let $\calO_{\ell} := \calO\otimes\F_{\ell}$ and $M := A[\ell](\bar{\Q})$.
Then $M$ is a free $\calO_{\ell}$-module of rank one, and $\Gal_{\Q}$ acts on $\calO_{\ell}$ by ring automorphisms (as studied in \S\ref{subsection: endomorphism field}) and on $M$ by $\F_{\ell}$-linear automorphisms. 
These actions satisfy $(a\cdot m)^\sigma = a^\sigma \cdot m^\sigma$ for all $\sigma \in \Gal_{\Q}, a\in \calO_{\ell}$ and $m\in M$.

\begin{lemma}\label{lemma: Galois actions on ring and p-torsion are incompatible}
Suppose that the $\Gal_{\Q}$-modules $\calO_{\ell}$ and $M$ are isomorphic.
Then $\ell \leq 3$.
\end{lemma}
\begin{proof}
This follows by comparing determinants.
On one hand, the $\Gal_{\Q}$-action on $\calO_{\ell}$ has determinant $1$. Indeed, the determinant of left/right multiplication by $b \in B$ acting on $B$ is the square of the reduced norm, so conjugation has determinant 1. 
On the other hand, the determinant of the $\Gal_{\Q}$-action on $M$ is the square of the mod $\ell$ cyclotomic character $\bar{\chi}_{\ell}$. 
This implies that $\bar{\chi}_{\ell}^2=1$, so $\Q(\zeta_{\ell}+\zeta_{\ell}^{-1}) = \Q$, so $\ell\leq 3$.
\end{proof}

\begin{remark}
When $\ell = 3$, we know of no examples of $\calO$-$\mathrm{PQM}$ surfaces over $\Q$ with $\calO_\ell \simeq M$ as $\Gal_{\Q}$-modules. Such examples do exist for $\ell = 2$; see \cite[Corollary 7.5]{LagaShnidman}. 
\end{remark}

\begin{lemma}\label{lem:orderp^2}
If $m\in M^{\Gal_{\Q}}$ is nonzero and $\ell\geq 5$, then $\calO_{\ell}\cdot m\subset M$ has order $\ell^2$.
\end{lemma}
\begin{proof}
By Lemmas \ref{lemma: classification submodules of Mat2} and \ref{lemma: classification submodules ramified case}, it suffices to show that $\calO_{\ell}\cdot m\neq M$.
But if $\calO_{\ell} \cdot m = M$, then $\calO_{\ell}\rightarrow M, x\mapsto x\cdot m$ is an isomorphism, contradicting Lemma \ref{lemma: Galois actions on ring and p-torsion are incompatible}.
\end{proof}

To analyze the case $\ell \mid \disc(B)$, we use the following theorem attributed to Ohta.

\begin{theorem}\label{thm:ohta}
 Let $F$ be a number field and let $A/F$ be an abelian variety with $\End(A) \simeq \calO$. Suppose $\calO$ is ramified at a prime $\ell$ and let $J \subset \calO$ be the maximal ideal above $\ell$. Then the composition of the Galois representation $\Gal_F \to \Aut_{\F_{\ell^2}}A[J] \simeq \F_{\ell^2}^\times$ with the norm $\F_{\ell^2}^\times \to \F_\ell^\times$ is equal to the mod $\ell$ cyclotomic character $\Gal_F \to \Aut(\mu_\ell) \simeq \F_\ell^\times$.
 \end{theorem}
\begin{proof}
See \cite[Proposition 4.6]{Jordan86}.   
\end{proof}
\begin{prop}\label{proposition: exclude torsion ramified case}
If $\ell\mid \disc(B)$ and $M^{\Gal_{\Q}}\neq 0$, then $\ell \leq 3$.
\end{prop}
\begin{proof}
Choose a nonzero $m\in M^{\Gal_{\Q}}$ and suppose that $\ell \geq 5$.
By the previous lemma, $\calO_{\ell}\cdot m$ is a proper submodule of $M$.
Therefore $\calO_{\ell}\cdot m = M[J]$ by Lemma \ref{lemma: classification submodules ramified case}.
Let $L/\Q$ be the endomorphism field of $A$.
Then the $\Gal_{\Q}$-action on $M$ restricts to a $\Gal_L$-action on $M[J]$ through elements of $\F_{\ell^2}^{\times}$ (after choosing an isomorphism $\calO_{\ell}/J \simeq \F_{\ell^2}$), giving a homomorphism $\epsilon \colon \Gal_L \rightarrow \F_{\ell^2}^{\times}$. 
Since $m$ is $\Gal_{\Q}$-invariant, the $\Gal_L$-action on $M[J]$ is trivial, so $\epsilon$ is trivial.
On the other hand, the composition  $N_{\F_{\ell^2}/\F_{\ell}}\circ \epsilon \colon \Gal_L \rightarrow \F_{\ell}^{\times}$ equals the mod $\ell$ cyclotomic character $\bar{\chi}_{\ell}$, by Theorem \ref{thm:ohta}.
It follows that $\bar{\chi}_{\ell}|_{\Gal_L} = 1$, or in other words $\Q(\zeta_{\ell})\subset L$.
Thus $\Gal(L/\Q)$ surjects onto $\Gal(\Q(\zeta_{\ell})/\Q)\simeq (\Z/\ell\Z)^{\times} \simeq \Z/(\ell-1)\Z$. 
Since $\Gal(L/\Q)$ is dihedral (Proposition \ref{proposition: galois group endomorphism field is dihedral if real place}), every nontrivial cyclic quotient of $\Gal(L/\Q)$ has order $2$, and we conclude that $\ell \leq 3$.
\end{proof}

We now treat the unramified case, using the following key linear-algebraic lemma, which we call the `torus trick'.
\begin{lemma}\label{lemma: reduction galois action generating torus}
Suppose that $\ell \nmid \disc(B)$.  Let $S\subset \calO_{\ell}$ be a $2$-dimensional semisimple commutative $\Gal_{\Q}$-stable subalgebra such that $S\cdot m = \calO_{\ell}\cdot m$ for some nonzero $m\in M^{\Gal_{\Q}}$.
Then every $\sigma \in \Gal_{\Q}$ acting trivially on $S$ also acts trivially on $\calO_{\ell}$.
\end{lemma}
\begin{proof}
Let $\sigma\in \Gal_{\Q}$ be an element acting trivially on $S$ and let $m \in M^{\Gal_{\Q}}\setminus \{0\}$ be an element such that $S\cdot m = \calO_{\ell}\cdot m$.
Let $k = \bar{\F}_{\ell}$.
It suffices to prove that $\sigma$ acts trivially on $\calO_k := \calO_{\ell} \otimes_{\F_{\ell}} k$.
The assumptions imply that $S_k \simeq k\times k$, and we may fix an isomorphism $\calO_k\simeq \Mat_2(k)$ of $k$-algebras such that $S_k$ is identified with the subalgebra of diagonal matrices of $\Mat_2(k)$. 
Lemma \ref{lemma: classification submodules of Mat2} and the fact that $S_k$ is $2$-dimensional shows that $\dim_k (S_k \cdot m) = \dim_k(\calO_k\cdot m)  = 2$.
Let $I = \{x\in \calO_k \mid x\cdot m=0\}$ be the annihilator of $m$, an ideal of $\calO_k$ of dimension $2$.
Using the analogue of Lemma \ref{lemma: classification submodules of Mat2} over $k$, such an ideal is necessarily of the form
\begin{align*}
    \left\{\begin{pmatrix} ax & bx \\ ay & by  \end{pmatrix} 
    \mid x,y\in k \right\}
\end{align*}
for some $a,b\in k$ which are not both zero.
The assumption that $S\cdot m = \calO_{\ell} \cdot m$ implies that $S_k \cap I =0$.
It follows that $a$ and $b$ must be nonzero and $\calO_k = S_k \oplus I$ as $\Gal_{\Q}$-modules.
Let $N\subset \calO_k$ be the subspace normalising but not centralising $S_k$.
Then the above calculation also shows that $N\cap I = 0$.
Moreover $N$ is $\Gal_{\Q}$-stable since $S$ is.
The relation $\calO_k = S_k \oplus I$ shows that $\sigma(x)-x\in I$ for all $x\in \calO_k$.
It follows that $\sigma(x)-x\in I \cap N = 0$ for all $x\in N$.
Since $\calO_k$ is spanned by $S_k$ and $N$, the claim follows. 
\end{proof}

\begin{prop}\label{proposition: large torsion unramified case means GL2-type}
Suppose that $\ell\nmid \disc(B)$ and $M^{\Gal_{\Q}}\neq 0$ and $\ell\geq 5$. 
Then $A$ is of $\GL_2$-type.
\end{prop}
\begin{proof}
We apply the torus trick using the distinguished quadratic subring $S \subset \calO$ of $A$ (Definition \ref{definition: distinguished quadratic subring}).
Write $S_{\ell} = S \otimes_\Z \F_{\ell}$.
Then $S_{\ell}\subset \calO_{\ell}$ is a commutative semisimple subalgebra since $S$ is unramified at $\ell$ by Proposition \ref{proposition: existence normalised torus}.
Suppose that $A$ is not of $\GL_2$-type.
Then $\Gal_{\Q}$ acts nontrivially on $S$ since $\End(A) = \Z$; let $K/\Q$ be the quadratic extension splitting this action.
We claim that the $\Gal_{K}$-action on $\calO_{\ell}$ is trivial.
Indeed, let $m\in M^{\Gal_{\Q}}$ be a nonzero element. 
By Lemma \ref{lemma: reduction galois action generating torus} it suffices to prove that $S_{\ell}\cdot m = \calO_{\ell}\cdot m$.
But the set $\{x\in S_{\ell} \mid x\cdot m=0\}$ is a proper $\Gal_{\Q}$-invariant ideal of $S_{\ell}$. 
Since the only such ideal is $0$ (using the fact that the $\Gal_{\Q}$-action on $S$ is nontrivial and $\ell\neq 2$), the map $S \cdot m\rightarrow \calO\cdot m$ is injective and hence by dimension reasons (and Lemma \ref{lem:orderp^2}) it must be surjective.
This proves that the $\Gal_{K}$-action on $\calO_{\ell}$ is trivial.
By Lemma \ref{lemma: reduction map is injective}, this even implies that that $\Gal_{K}$-action on $\calO$ is trivial.
We conclude that the quadratic field $K$ is the endomorphism field of $A$, hence $A$ is of $\GL_2$-type by Lemma \ref{lemma: A GL2-type iff endo field quadratic}, contradiction.
\end{proof}

\section{Proof of Theorems \ref{thm:optimalbound} and \ref{thm:QMmazur_groups}: eliminating groups of order \texorpdfstring{$2^i3^j$}{2i3j}}

Let $A/\Q$ be an $\calO$-$\mathrm{PQM}$ surface.
By Theorem \ref{thm:QMmazur ptorsion}, we have $\#A(\Q)_{\tors} = 2^i3^j$ for some $i,j\geq 0$.
Since $A$ has potentially good reduction, local methods show that $2^i3^j\leq 72$ \cite[Theorem 1.4]{CX08}.
In this section, we will improve this bound and constrain the group structure of $A(\Q)_{\mathrm{tors}}$ as much as possible using the $\calO$-action on $A_{\Qbar}$.
We may assume $A$ is not of $\GL_2$-type since we have already proven Theorem \ref{thm:gl2type classification}.

For each prime $p$, there exists a totally ramified extension $K/\Q_p$ such that $A_K$ has good reduction (Lemma \ref{lemma:PQMreduction}).
The special fiber of the N\'eron model of $A_K$ is an abelian surface over $\F_p$ which we denote by $A_p$. 
We call $A_p$ \defi{the good reduction of $A$ at $p$}, though it is only uniquely determined up to twists (since a different choice of totally ramified extension $K'$ would give rise to a possibly non-isomorphic twist of $A_p$).

Lemma \ref{lem:reduction is injective} shows that the prime-to-$p$ subgroup of $A(\Q)_{\mathrm{tors}}$ injects into $A_p(\F_p)$.
Moreover $\End(A_{\Qbar})\subset \End(A_{\bar\F_p})$ hence $A_p$ is $\bar\F_p$-isogenous to the square of an elliptic curve $E/\bar\F_p$ by Proposition \ref{prop: QM splits mod p}, so its isogeny class is rather constrained.
This leads to the following slight strengthening of \cite[Theorem 1.4]{CX08} in our case:

\begin{prop}\label{prop:naivebound}
We have $\#A(\Q)_{\tors}= 2^i3^j$ for some $i \in \{0,1,2,3,4\}$ and $j \in \{0,1,2\}$. Moreover, $\#A(\Q)_{\tors} \leq 48$.
\end{prop}
\begin{proof}
 By the above remarks, to bound the prime-to-$2$ (resp.\ prime-to-3) torsion, it is enough to bound $X(\F_2)[3^\infty]$ (resp.\ $X(\F_3)[2^\infty]$), as $X$ varies over all abelian surfaces over $\F_2$ (resp.\ $\F_3$) that are geometrically isogenous to the square of an elliptic curve. For this it is enough to compute $\max_X\gcd(f_X(1),3^{100})$ (resp.\ $\max_X\gcd(f_X(1),2^{100})$), where $f_X$ is the $L$-polynomial of $X$ and the maximum is over all the aforementioned  isogeny classes. This computation is easily done with the help of the LMFDB's database of isogeny classes of abelian varieties over finite fields \cite{lmfdb}, and the conclusion is the first sentence of the proposition. 

 The second sentence is equivalent to the claim that $\#A(\Q)_{\tors}$ cannot equal $144$ nor $72$. We cannot have $144$ since $\#A_5(\F_5) \leq 100$, and we cannot have $72$ since the only isogeny class of abelian surfaces $X/\F_5$ with $72 \mid \#X(\F_5)$ (which has LMFDB label \href{https://www.lmfdb.org/Variety/Abelian/Fq/2/5/f_q}{$2.5.f_q$}) is not geometrically isogenous to a square of an elliptic curve.  
\end{proof}

The remainder of the proof of Theorems \ref{thm:optimalbound} and \ref{thm:QMmazur_groups} will be similar (but more difficult) to that of \ref{prop:naivebound}, using the good reduction model $A_p$ at various primes $p$ and the $\calO$-action.
In what follows, we will freely use the Honda-Tate computations conveniently recorded in the LMFDB \cite{lmfdb}, so the careful reader will want to follow along in a web browser.  We use the LMFDB's method of labeling isogeny classes, e.g. \href{https://www.lmfdb.org/Variety/Abelian/Fq/2/5/d_e}{$2.5.d_e$} is an isogeny class of abelian surfaces over $\F_5$ with label $d_e$.

\subsection{Torsion constraints arising from the endomorphism field}
Before analyzing specific groups, we state the following useful proposition, which uses techniques similar to the proof of Theorem \ref{theorem: reduction to GL2-type}, including the torus trick.

\begin{prop}\label{prop: D3,D6 case has little 2-torsion, D2,D4 case has little 3-torsion}
Let $G$ be the Galois group of the endomorphism field $L/\Q$.
\begin{enumalph}
    \item If $G \simeq D_3$ or $D_6$, then $A[2](\Q)\subset \Z/2\Z$.
    If in addition $A[2](\Q) = \Z/2\Z$, then $A[2]\simeq \calO/2\calO$ as $\Gal_{\Q}$-modules or $2\mid \disc(B)$.
    \item If $G \simeq D_2$ or $D_4$, then $A[3](\Q)\subset \Z/3\Z$.
    If in addition $A[3](\Q) = \Z/3\Z$, then $A[3]\simeq \calO/3\calO$ as $\Gal_{\Q}$-modules or $3\mid \disc(B)$.
\end{enumalph}
\end{prop}
\begin{proof}
    \begin{enumalph}
    \item Let $S\subset \calO$ be the distinguished quadratic subring of $A$ (Definition \ref{definition: distinguished quadratic subring}).
    By Proposition \ref{proposition: existence normalised torus}, $S \simeq \Z[\omega]$ where $\omega^2+ \omega+1=0$.
    Let $K/\Q$ be the quadratic field trivializing the Galois action on $S$, so $\End(A_K) = S$.
    Let $S_2 := S\otimes \F_2$ and $\calO_2 := \calO\otimes \F_2$.
    If $A[2]\simeq \calO_2$ as $\Gal_{\Q}$-modules, then $A[2](\Q) \simeq (\calO/2\calO)^{\Gal_{\Q}}$ is isomorphic to $\Z/2\Z$ by Theorem \ref{theorem: fixed points of O/N}, so indeed $A[2](\Q) \subset \Z/2\Z$ in this case.
    We may therefore assume that $A[2](\Q) \not\simeq \calO_2$ in what follows.
    
    It suffices to show that if there exists a nonzero $m \in A[2](\Q)$, then $A[2](\Q)$ has order $2$.
    By the classification of $\calO_2$-submodules of $A[2]$ of \S\ref{subsection: linear algebra in calO mod p} and the fact that $\calO_2$ is not isomorphic to $A[2]$, the submodule $\calO_2 \cdot m \subset A[2]$ has order $4$.
    Since $S_2 \simeq \F_4$ has no $\Gal_{\Q}$-stable nonzero proper ideals, the map $S_2\rightarrow A[2], x\mapsto x\cdot m$ is injective, hence $S_2\cdot m \subset \calO_2\cdot m$ has order $4$ too.
    Therefore $S_2\cdot m =\calO_2 \cdot m$.
    Suppose first that $2 \nmid \disc(B)$.
    We can then apply Lemma \ref{lemma: reduction galois action generating torus} to conclude that $\Gal_K$ acts trivially on $\calO_2$. 
    Since $\Gal_K$ acts on $\calO_2$ through $\Gal(L/K)\simeq C_3$ or $C_6$, this contradicts Lemma \ref{lemma: reduction map is injective} and proves the second claim of (a).
    It remains to consider the case $2\mid \disc(B)$.
    In that case there exists a unique proper nonzero left ideal $J$ of $\calO_2$, and $A[J]$ is the unique nonzero proper $\calO_2$-submodule of $A[2]$ (Lemma \ref{lemma: classification submodules ramified case}).
    It follows that $S_2\cdot m = \calO_2\cdot m = A[J]$.
    Since $A[2]\not\simeq \calO_2$ as $\Gal_{\Q}$-modules, no element of $A[2](\Q)$ is an $\calO_2$-generator, so $A[2](\Q)= A[J](\Q)$.
    On the other hand, the equality $S_2\cdot m = A[J]$ shows that $S_2\simeq A[J]$ as $\Gal_{\Q}$-modules. 
    Since $\Gal_{\Q}$ acts nontrivially on $S_2 = \F_4$, $A[J](\Q) = A[2](\Q)$ has order $2$.

    \item The argument is very similar to the proof of (a), using that in the $D_4$ case, the distinguished quadratic subring $\Z[i]$ is unramified at $3$.
    In the $D_2$ case, the distinguished quadratic subring might be ramified at $3$, but by Lemma \ref{lemma: D2 subgroup Aut(O)} there exist three squarefree integers $m,n,t$ and embeddings of $\Z[\sqrt{m}]$, $\Z[\sqrt{n}]$ and $\Z[\sqrt{t}]$ into $\Gal_{\Q}$ whose image is $\Gal_{\Q}$-stable.
    Since $t = -mn$ up to squares, at least one of these three subrings is unramified at $3$, and the argument of (a) can be carried out using this subring.
    \end{enumalph}
\end{proof}
\subsection{Groups of order \texorpdfstring{$48$}{48}}

\begin{lemma}\label{lem:direct sum}
Let $E$ be an elliptic curve over the finite field $\F_{p^n}$, and assume either that $E$ is ordinary or that $n = 1$. Then any abelian surface $X/\F$ isogenous to $E^2$ is isomorphic to a product of elliptic curves over $\F$.   
\end{lemma}

\begin{proof}
Let $\pi \in \End(E)$ be the Frobenius. Replacing $E$ by an isogenous elliptic curve, we may assume that $\End(E) = \Z[\pi]$ \cite[\S7.2-7.3]{JKPRST}.
By \cite[Theorem 1.1]{JKPRST}, the functor $X \mapsto \Hom(X,E)$ is an equivalence between the category of abelian varieties isogenous to a power of $E$ and isomorphism classes of finitely generated torsion-free $\End(E)$-modules. Since $\End(E)$ is an order in a quadratic field, any finitely generated torsion-free $\End(E)$-module is a direct sum of rank $1$ modules \cite[Theorem 3.2]{JKPRST}, so the lemma follows. 
\end{proof}

\begin{lemma}\label{lemma: torsion order 16}
If $G\subset A(\Q)_{\tors}$ is a subgroup of order $16$, then $G$ is isomorphic to $(\Z/4\Z)^2$ or $(\Z/2\Z)^2 \times \Z/4\Z$.
\end{lemma}

\begin{proof}
There is a unique isogeny class of abelian surfaces $X$ over $\F_3$ with $16\mid \#X(\F_3)$, namely the square of the elliptic curve $E/\F_3$ with $\End_{\F_p}(E) \simeq \Z[\sqrt{-3}]$ and $\#E(\F_3) = 4$. By Lemma \ref{lem:direct sum},  $A_p$ is isomorphic to a product of two elliptic curves both of which have four $\F_3$-rational points. 
Since such an elliptic curve has its group of $\F_3$-points isomorphic to either $\Z/4\Z$ or $(\Z/2\Z)^2$, $A_p(\F_3)$ is isomorphic to $(\Z/4\Z)^2$ or $(\Z/4\Z)\times(\Z/2\Z)^2$ or $(\Z/2\Z)^4$. 
By Proposition \ref{prop: no (Z/2)^4}, the latter cannot happen. The lemma now follows since $A(\Q)[16] \hookrightarrow A_p(\F_3)$.
\end{proof}

\begin{prop}\label{prop:no48}
$\#A(\Q)_{\tors} < 48$.  
\end{prop} 

\begin{proof} 
By Proposition \ref{prop:naivebound} it is enough to show that $A(\Q)_{\tors} \neq 48$. Assume for the sake of contradiction that $\#A(\Q)_{\mathrm{tors}} = 48$. The reduction $A_5/\F_5$ must then be in the isogeny class \href{https://www.lmfdb.org/Variety/Abelian/Fq/2/5/d_e}{$2.5.d_e$}.
We see that $\End^0((A_p)_{\F_{5^n}})$ contains a quaternion algebra if and only if $3$ divides $n$.
Therefore the Galois group of the endomorphism field of $A$ has order divisible by $3$, so by Proposition \ref{proposition: galois group endomorphism field is dihedral if real place} must be $D_3$ or $D_6$.
Proposition \ref{prop: D3,D6 case has little 2-torsion, D2,D4 case has little 3-torsion} then implies $A[2](\Q)\subset \Z/2\Z$, contradicting the fact that $A[2](\Q)$ has size $\geq 4$ (Lemma \ref{lemma: torsion order 16}).
\end{proof}

\subsection{Groups of order \texorpdfstring{$36$}{36}}
\begin{lemma}\label{lemma: torsion order 36}
If $36 \mid \#A(\Q)_\mathrm{tors}$, then $A(\Q)_\mathrm{tors} \simeq (\Z/6\Z)^2$. 
\end{lemma}
\begin{proof}
Over $\F_5$ there is exactly one isogeny class of abelian surface $X$ with $36 \mid \#X(\F_5)$ and whose geometric endomorphism algebra contains a quaternion algebra, namely \href{https://www.lmfdb.org/Variety/Abelian/Fq/2/5/a_k}{$2.5.a_k$}, which is isogenous to the square of an elliptic curve.
Thus the reduction $A_5$ is isomorphic to a product of two elliptic curves (Lemma \ref{lem:direct sum}).  Every elliptic curve in this isogeny class has $E(\F_5) \simeq \Z/6\Z$, hence $A_5(\F_5) \simeq (\Z/6\Z)^2$. 
\end{proof}

\begin{prop}\label{prop:no36}
$\#A(\Q)_{\tors} < 36$.
\end{prop}

\begin{proof}
By Proposition \ref{prop:no48} and Proposition \ref{prop:naivebound}, it is enough to show that $A(\Q)_{\tors}$ does not have order $36$. By Lemma \ref{lemma: torsion order 36} such an $A$ would have $A(\Q)_{\tors} \simeq (\Z/6\Z)^2$. 
By Proposition \ref{prop: D3,D6 case has little 2-torsion, D2,D4 case has little 3-torsion}, $A$ cannot have endomorphism field $D_n$ for every $n\in \{2,3,4,6\}$ so $A$ has $\GL_2$-type, which we have also already ruled out. 
\end{proof}

It follows that $\#A(\Q)_{\tors} \leq 24$. Before we show that this inequality is strict, we rule out the existence of rational points of order $9$ and $8$.

\subsection{Rational points of order \texorpdfstring{$9$}{9}}

\begin{prop}\label{prop:no Z/9}
$A(\Q)_{\tors}$ contains no elements of order $9$.
\end{prop}
\begin{proof}
Suppose $A(\Q)$ has a point of order $9$. Then the reduction $A_2/\F_2$ must live in the isogeny class
\href{https://www.lmfdb.org/Variety/Abelian/Fq/2/2/a_e}{$2.2.a_e$} or \href{https://www.lmfdb.org/Variety/Abelian/Fq/2/2/b_b}{$2.2.b_b$}. The latter has commutative geometric endomorphism algebra, so cannot be the reduction of a $\calO$-$\mathrm{PQM}$ surface by Proposition \ref{prop: QM splits mod p}.
The former is the isogeny class of the square of an elliptic curve $E$ over $\F_2$ with $\#E(\F_2) = 3$, so by Lemma \ref{lem:direct sum} we have $A_2(\F_2) \simeq (\Z/3\Z)^2$. 
\end{proof}

\subsection{Rational points of order \texorpdfstring{$8$}{8}}

\begin{prop}\label{prop:no8}
$A(\Q)_{\tors}$ contains no elements of order $8$.
\end{prop}
\begin{proof}
Suppose otherwise. The reduction $A_3/\F_3$ must be in the isogeny class \href{https://www.lmfdb.org/Variety/Abelian/Fq/2/3/a_ac}{$2.3.a_c$}, which is simple with endomorphism algebra $\Q(\zeta_8) = \Q(\sqrt{2}, \sqrt{-2})$. (It cannot be in the isogeny class  \href{https://beta.lmfdb.org/Variety/Abelian/Fq/2/3/a_g}{$2.3.a_g$} by the proof of Lemma \ref{lemma: torsion order 16}.)  Since $\#A_3(\F_3) = 8$, we must have $A_3(\F_3) = \Z/8\Z$. This eliminates the possibility that $A(\Q)$ contains a prime-to-3 subgroup any larger than $\Z/8\Z$. Note also that $\#A_3(\F_9) = 64$ and $A$ is isomorphic to a product of ordinary elliptic curves over $\F_9$ by Lemma \ref{lem:direct sum}, at least one of which has $E(\F_9) \simeq \Z/8\Z$.  It follows that the $\F_2$-dimension of $A_3[2](\F_9)$ is at most $3$, and in particular not all $2$-torsion points are defined over $\F_9$. On the other hand, all endomorphisms of $(A_3)_{\bar \F_3}$ are defined over $\F_9$, so we conclude by Lemmas \ref{lemma: classification submodules of Mat2} and \ref{lemma: classification submodules ramified case} that the $\calO/2\calO$-module generated by any $\F_9$-rational point of order $2$ has order $4$.

Suppose first that $2$ divides $\disc(B)$. Then the aforementioned $\calO$-module must be $A[J]$, where $J$ is the ideal in $\calO$  such that $J^2 = 2\calO$ (see \S\ref{subsection: linear algebra in calO mod p}).  Let $t \in J$ be any element not in $2\calO$.  Then over $\F_9$ we have an exact sequence 
\[0 \to A_3[J] \to A_3[2] \to A_3[J] \to 0\]
with the last map being multiplication by $t$. 
Let $P \in A_3[4](\F_9)$ be a point of order $4$. Without loss of generality we may assume $Q = tP$ has order $2$ (if not, just replace $P$ by $tP$) and $Q \notin A_3[J]$. Then we've seen that $\calO \cdot Q \neq A_3[2]$, so $\calO\cdot Q = A_3[J]$ but this contradicts $Q \notin A_3[J]$.

Now suppose that $2$ does not divide $\disc(B)$ so that $\calO \simeq \mathrm{Mat}_2(\F_2)$. Let $L/\Q$ be the endomorphism field. If $\Gal(L/\Q) \simeq D_2$ then at least one of the quadratic subfields of $L$ is not inert at $3$. So $\End_{\F_3}(A_3)$ must contain a quadratic order $S$ in $\Z[i]$ or $\Z[\sqrt{2}]$ or in $\Z[\sqrt{-2}]$. But we saw in Lemma \ref{lemma: centraliser of order 2 subgroup of automorphism group} that $S$ contains $\Z[\sqrt{m}]$ with $m$ squarefree. So $S$ {\it is} $\Z[i]$ or $\Z[\sqrt{2}]$ or $\Z[\sqrt{-2}]$. In all cases there exists $t \in S$ such that $t^2S = 2S$, and so we have an endomorphism (defined over $\F_3$) which behaves like $\sqrt2$ on $A_3[2]$.
But we also have a rational point $P$ of order 4. Without loss of generality the orders of $tP$ and $t^2P$ are both $2$. But $t^2P \neq tP$, so $\dim_{\F_2} A_3[2](\F_3) > 1$, which contradicts $A_3(\F_3) \simeq \Z/8\Z$.
The case $\Gal(L/\Q) = D_4$ does not happen when $\disc(B)$ is odd by Lemma \ref{lemma: D4 subgroup Aut(O)}, so we consider the case where $\Gal(L/\Q)$ is $D_3$ or $D_6$. 
By Proposition \ref{prop: D3,D6 case has little 2-torsion, D2,D4 case has little 3-torsion}(a), $A[2] \simeq \calO/2\calO$ as $\Gal_{\Q}$-modules.
But then $A_3[2] \simeq \calO/2\calO$ as $\Gal_{\F_3}$-modules, contradicting the fact that $A_3[2](\F_3)$ contains no $\calO/2\calO$-generator.

We are left to consider the case $\Gal(L/\Q) = D_1 = C_2$, i.e.\ the $\GL_2$-type case, which we have already treated in Proposition \ref{prop:gl2 constrains}. 
\end{proof}

\subsection{Groups of order \texorpdfstring{$24$}{24}}

If $A(\Q)_{\tors}$ has order $24$, then by Proposition \ref{prop:no8}, the group structure is either $(\Z/2\Z)^3 \times \Z/3\Z$ or $\Z/2\Z \times \Z/4\Z \times \Z/3\Z$. We show below that in fact neither can occur. First we gather some facts common to both cases. 

\begin{lemma}\label{lem:24 facts}
Suppose $\#A(\Q)_{\mathrm{tors}} = 24$, and let $L/\Q$ be the endomorphism field of $A$. Then
\begin{enumalph}
    \item $\Gal(L/\Q)$ is isomorphic to $D_2$ or $D_4$,
    \item $\Q(\zeta_3) \subset L$, and
    \item if $\Gal(L/\Q) \not\simeq D_4$ then $A$ has unipotent rank $1$ over $\Q_3$ $($in the terminology of \S\ref{subsec: neron models of PQM surfaces}$)$.
\end{enumalph}
\end{lemma}
\begin{proof}
Since $A$ is not of $\GL_2$-type,  Proposition \ref{prop: D3,D6 case has little 2-torsion, D2,D4 case has little 3-torsion} implies that $\Gal(L/\Q)$ is isomorphic to $D_2$ or $D_4$, proving $(a)$.

Checking isogeny classes over $\F_5$, we see that the reduction $A_5$ is in the isogeny class \href{https://www.lmfdb.org/Variety/Abelian/Fq/2/5/a_ac}{$2.5a_{ac}$}; the isogeny class \href{https://www.lmfdb.org/Variety/Abelian/Fq/2/5/d_e}{$2.5d_e$} is ruled out since it only acquires QM over $\F_{5^3}$, which is not compatible with $(a)$. The fact that $\#A_5(\F_{25})[3^\infty] = 9$ shows that the point of order $3$ in $A(\Q)$ is not an $\calO$-module generator of $A[3]$ (since the $\calO$-action on $A_5$ is defined over $\F_{25}$). By Proposition \ref{prop: D3,D6 case has little 2-torsion, D2,D4 case has little 3-torsion}, we deduce that the quaternion algebra $B$ is ramified at $3$. Since $A[3](\Q)$ has a rational point, it follows from Theorem \ref{thm:ohta} that $\Q(\sqrt{-3}) = \Q(\zeta_3) \subset L$, proving $(b)$.

Since $3$ ramifies in $L$, $A$ has bad reduction over $\Q_3$ by Proposition \ref{prop: Silverberg result endo field}. If $A[2](\Q) \simeq (\Z/2\Z)^3$ then $A$ achieves good reduction over every ramified quadratic extension of $\Q_3$ by Proposition \ref{prop: (Z/2)^3 implies good red quad extension}. 
If $A/\Q_3$ has totally additive reduction, then the quadratic twist of $A$ by $\Q(\sqrt{3})$, say,  will have good reduction at $3$ by Lemma \ref{lemma: quadratic twist away tot add primes with quad good red field}. But quadratic twisting does not change the endomorphism field by Lemma \ref{lemma: quadratic twisting doesnt change endo field}, so any quadratic twist of $A$ must have endomorphism field which contains $\Q(\sqrt{-3})$ and hence must have bad reduction at $3$.  We conclude that $A$ must have unipotent rank $1$ over $\Q_3$ by Proposition \ref{prop:GL2totadditive}.

If $A[2](\Q) \simeq (\Z/2\Z)^2$ and $\Gal(L/\Q) \not\simeq D_4$, then $\Gal(L/\Q) \simeq D_2$ and so $L/\Q$ is a biquadratic field containing $\Q(\zeta_3)$. It follows that $A$ has all of its endomorphisms defined over  $\Q_3^{\mathrm{nr}}(\zeta_3)$. If $A$ still has bad reduction over $\Q_3(\zeta_3)$, then it must have totally additive bad reduction (since it has QM after enlarging the residue field) by Proposition \ref{prop:GL2totadditive}, and we obtain a contradiction with Proposition \ref{prop:gl2 type lorenzini variant} and the fact that $A$ has a point of order $4$.
Thus, $A$ attains good reduction over $\Q_3(\zeta_3)$, and arguing as above, we conclude that $A$ has unipotent rank $1$ over $\Q_3$. 
\end{proof}

\begin{prop}\label{prop: no 24a}
$A(\Q)_{\tors} \not\simeq (\Z/2\Z)^3 \times \Z/3\Z$.
\end{prop}

\begin{proof}
Assume otherwise. 
Theorem \ref{theorem: fixed points of O/N} and Lemma \ref{lem:linalg} show that the endomorphism field $L/\Q$ has Galois group $\Gal(L/\Q) \simeq D_2$. 

First assume there exists a prime $p > 3$ of bad reduction for $A$. By Theorem \ref{thm:lorenzini}, $A$ must have unipotent rank 1 over $\Q_p$, and hence $p$ must ramify in $L$ by Proposition \ref{prop:GL2totadditive}.  Next, recall  that there are three $\Gal_\Q$-stable quadratic subfields of $B$, one of which is imaginary. Let $L_1$, $L_2$, and $L_3$ be the corresponding quadratic subfields of $L$, labeled so that $B^{\Gal_{L_1}}$ is imaginary quadratic. Since $L$ is biquadratic, exactly one of the $L_i$ must be unramified over $\Q_p$. Since $A$ has unipotent rank 1, it must be $L_1$ (by Proposition \ref{prop:GL2totadditive}). But by Lemma \ref{lem:24 facts}(b) we have $\Q(\zeta_3) \subset L$ and $\Q(\zeta_3)$ is also unramified at $p$, so 
$L_1 = \Q(\sqrt{-3})$.  Now, $A/\Q_3$ has unipotent rank 1 by Lemma \ref{lem:24 facts}(c). As above, Proposition \ref{prop:GL2totadditive} implies that the unique sub-extension $L_i$ unramified at $3$ must be $L_1$. This contradicts $L_1 = \Q(\sqrt{-3})$. 

Thus, it remains to consider the possibility that $A$ has good reduction outside $\{2,3\}$. This forces the endomorphism field to be unramified outside $\{2,3\}$.
Moreover, $A$ has unipotent rank $1$ reduction over $\Q_3$, so $L$ must contain an imaginary quadratic subfield that is unramified at $3$.  Hence $L$ is isomorphic to $\Q(\sqrt{-3},i)$ or $\Q(\sqrt{-3}, \sqrt{-2})$.  We also know that $B^{\Gal_{\Q(\sqrt{-3})}}$ is a real quadratic field, and $L_1$ is either $\Q(i)$ or $\Q(\sqrt{-2})$. 

Over $\F_7$, there are two possible isogeny classes: \href{https://www.lmfdb.org/Variety/Abelian/Fq/2/7/a_ac}{$2.7a_{ac}$} and \href{https://www.lmfdb.org/Variety/Abelian/Fq/2/7/i_be}{$2.7i_{be}$}. Since $7$ is inert in $L_1$, $L$ does not split completely at $7$. The isogeny class is therefore not $2.7i_{be}$, since all its endomorphisms are defined over $\F_7$, hence the isogeny class is $2.7a_{ac}$. Thus $\End^0(A_7) \simeq \Q(\sqrt{-3}) \times \Q(\sqrt{-3})$. Since $7$ splits in $\Q(\sqrt{-3})$, we see that $B^{\Gal_{\Q(\sqrt{-3})}} = \Q(\sqrt{-3})$, which shows that $L_1 = \Q(\sqrt{-3})$, contradicting what was said above.
\end{proof}

\begin{prop}\label{prop: no 24b}
If $A(\Q)_{\tors} \not\simeq (\Z/2\Z) \times (\Z/4\Z) \times (\Z/3\Z)$. 
\end{prop}

\begin{proof}
First suppose $G \simeq D_4$, so that the distinguished subring $S$ of Definition \S\ref{definition: distinguished quadratic subring} is isomorphic to $\Z[i]$.
Then $2\mid \disc(B)$ by Lemma \ref{lemma: D4 subgroup Aut(O)}.
Since $B$ is ramified at $2$ and $3$ and $A(\Q)$ contains points of order $4$ and $3$, we see that $L$ contains both $\Q(i)$ and $\Q(\zeta_3)$, by Theorem \ref{thm:ohta}. Over one of these two quadratic subfields, the $\Gal_\Q$-action on $S = \Z[i]$ trivializes. Indeed, the $\Gal_\Q$-action on $\Z[i]$ cannot be trivialized by the third quadratic subfield $\Q(\sqrt{3})$ of $L$, by Proposition \ref{proposition: endomorphism rings of absolutely simple abelian surfaces over R}.  
Looking over $\F_5$ we see that $\Q(i)$ could only trivialize a ring isomorphic to $\Z[\sqrt{3}]$.  
Looking over $\F_7$ we see that $\Q(\zeta_3)$ could only trivialize a ring isomorphic to $\Z[\sqrt{-3}]$.
So neither trivialize $\Z[i]$, and we have a contradiction. 

So we may now assume that $G \simeq D_2$. 
Arguing as above, we may also assume that $L$ does not contain $\Q(i)$. 
We know $A$ has unipotent rank 1 reduction over $\Q_3$ by Lemma \ref{lem:24 facts}(c).  It also has unipotent rank 1 reduction at all bad primes $p > 3$, by Theorem \ref{thm:lorenzini}. By Proposition \ref{prop:GL2totadditive},  the imaginary quadratic subfield $L_1 \subset L$ that trivializes the distinguished imaginary quadratic subring of $\calO$ is unramified outside $\{2\}$. Since $L_1 \neq \Q(i)$, we must have $L_1 = \Q(\sqrt{-2})$, but this field does not embed in $B$ (which is ramified at $3$), giving a contradiction.
\end{proof}

As a corollary, we are now able to finish the proofs of  Theorems \ref{thm:optimalbound} and \ref{thm:QMmazur_groups}.
\begin{proof}[Proof of Theorem $\ref{thm:optimalbound}$]
Propositions \ref{prop:no8}, \ref{prop: no 24a}, and \ref{prop: no 24b} show that $\#A(\Q)_{\tors} < 24$. Hence $\#A(\Q)_{\tors} \leq 18$.
\end{proof}

By the results of this section and the previous one, the group $A(\Q)_{\tors}$ has order $2^i3^j \leq 18$ and does not contain any subgroup of the form $\Z/8\Z$, $\Z/9\Z$, or $(\Z/2\Z)^4$.
We deduce the following result, which is equivalent to Theorem \ref{thm:QMmazur_groups}.

\begin{theorem}\label{thm:halftime}
Let $A/\Q$ be an abelian surface such that $\End(A_{\Qbar})$ is a maximal order in a non-split quaternion algebra. Then $A(\Q)_{\tors} = A[12](\Q)$ and $\#A(\Q)_{\tors}\leq 18$. Moreover, $A(\Q)_{\tors}$ does not contain a subgroup isomorphic to $(\Z/2\Z)^4$. In other words, $A(\Q)_{\tors}$ is isomorphic to one of the groups
\begin{align*}
    \{1\},\Z/2,\Z/3,\Z/4, (\Z/2\Z)^2, \Z/6, (\Z/2\Z)^3, \Z/2\Z \times \Z/4\Z,(\Z/3\Z)^2,\\ \Z/12,   \Z/2\Z\times \Z/6\Z, (\Z/2\Z)^2 \times \Z/4\Z, (\Z/4\Z)^2,\Z/3\Z\times \Z/6\Z.
\end{align*}    
\end{theorem}

Not all of the groups above are known to be realized as $A(\Q)_{\tors}$ for some $\calO$-$\mathrm{PQM}$ surface $A/\Q$. However, all groups that have been realized (including the largest one of order 18) have been realized in the family of bielliptic Picard Prym surfaces \cite{LagaShnidman}. It would be interesting to systematically analyze rational points on Shimura curves of small discriminant and with small level structure, to try to find more examples. It would also be interesting to see which groups can be realized by Jacobians, which is the topic we turn to next.

\section{Proof of Theorem \ref{thm:PQM Jacobians}: PQM Jacobians}
In this section, we consider $\calO$-$\mathrm{PQM}$ surfaces $A/\Q$ equipped with a principal polarization. Since $A$ is geometrically simple, there exists an isomorphism of polarized surfaces $A \simeq \Jac(C)$, where $C$ is a smooth projective genus two curve over $\Q$ \cite[Theorem 3.1]{Sekiguchi-onfieldsofrationality}. To emphasize this, we use the letter $J$ instead of $A$. The goal of this section to prove some additional constraints on the torsion group $J(\Q)_{\tors}$, i.e.\ we prove Theorem \ref{thm:PQM Jacobians}.

\begin{lemma}\label{lem:principal polarization quadratic field}
Let $M$ be the imaginary quadratic subfield of $\End^0(A_{\bar{\Q}})$ corresponding to a principal polarization on $J$ under Corollary \ref{cor:quadsubring via polarization}.
Then $M\simeq \Q(\sqrt{-D})$, where $D = \disc(B)$.   
\end{lemma}
\begin{proof}
This is a direct consequence of the relation \eqref{equation: degree polarization vs norm positive involution} of Proposition \ref{prop:relation polarizations positive involutions}.
\end{proof}

\begin{lemma}\label{lemma:endofield constraint jacobian}
The endomorphism field $L/\Q$ has Galois group $D_1 = C_2$ or $D_2 = C_2 \times C_2$.
\end{lemma}
\begin{proof}
See \cite[Theorem 3.4 A(1)]{DR04}.
\end{proof}

\begin{prop}\label{prop:no18}
$\#J(\Q)_{\tors} < 18$.
\end{prop}

\begin{proof}
By Theorem \ref{thm:QMmazur_groups}, we need only exclude $(\Z/2\Z) \times (\Z/3\Z)^2$. 
By Proposition \ref{prop: D3,D6 case has little 2-torsion, D2,D4 case has little 3-torsion}(b) and Lemma \ref{lemma:endofield constraint jacobian}, the endomorphism field of $A$ would be a $C_2$-extension.
In other words, $A$ is of $\GL_2$-type, but this contradicts Theorem \ref{thm:gl2type classification}.
\end{proof}

Finally, we rule out the group $(\Z/2\Z)^3$ from appearing in $J[2](\Q)$. We have already proven this when $J$ is of $\GL_2$-type (Proposition \ref{prop: GL2 type no (Z/2)^3 torsion}), so it remains to consider the case $\Gal(L/\Q) \simeq C_2 \times C_2$. We deduce this from the following more general result.

\begin{prop}\label{prop:(Z/2)^3 constraints}
Suppose that $A/\Q$ is $\calO$-$\mathrm{PQM}$, has $C_2\times C_2$ endomorphism field and has $A[2](\Q)\simeq (\Z/2\Z)^3$.
Let $d$ be the degree of the unique primitive polarization of $A$.
Then $2\mid \disc(B)$ and there exists an integer $m \equiv 1 \mod 4$ such that $\disc(B)$ and $d m$ agree up to squares.
In particular, $d$ is even and $A$ is not a Jacobian.
\end{prop}
\begin{proof}
    Let $L/\Q$ be the endomorphism field of $A$ with Galois group $G$.
    By Lemma \ref{lem:linalg}, there exists an $\Q$-rational $\calO/2\calO$-generator $P\in A[2](\Q)$, hence $A[2]\simeq \calO/2\calO$ as $\Gal_{\Q}$-modules.
    Therfore the $G$-action on $\calO/2\calO$ has $(\Z/2\Z)^3$ fixed points.
    By Lemma \ref{lemma: C2xC2 with (Z/2)^3 fixed points}, $2\mid \disc(B)$ and there exist positive integers $m,n$ with $m\equiv 1\mod 4$ and $n \equiv 3 \mod 4$ such that the three $\Gal_{\Q}$-stable quadratic subfields of $B$ are $\Q(\sqrt{-m}), \Q(\sqrt{n})$ and $\Q(\sqrt{mn})$.
    Under Corollary \ref{cor:quadsubring via polarization}, the unique primitive polarization of $A$ corresponds to the subfield $\Q(\sqrt{-m})$, and the relation \eqref{equation: degree polarization vs norm positive involution} of Proposition \ref{prop:relation polarizations positive involutions} shows that $d \disc(B)$ and $m$ agree up to squares.
    In other words, $\disc(B)$ and $dm$ agree up to squares. 
    Since $\disc(B)$ is even and squarefree and $m$ is odd, $d$ must be even too. 
\end{proof}


\begin{proof}[Proof of Theorem $\ref{thm:PQM Jacobians}$]
Combine Theorem \ref{thm:QMmazur_groups} and Propositions \ref{prop:no18} and \ref{prop:(Z/2)^3 constraints}.   
\end{proof}

In Table \ref{PQMequationstable} we give some examples of Jacobians with non-trivial torsion subgroups  and $\calO_D$-PQM, where $\calO_D$ is a maximal quaternion order of discriminant $D$. These were found by computing the relevant covers of Shimura curves of level 1 and their full Atkin-Lehner quotients and then substituting into the Igusa-Clebsch invariants in \cite[Appendix B]{LY20}. The torsion and endomorphism data can be independently verified using MAGMA\footnote{\url{https://github.com/ciaran-schembri/QM-Mazur}}.

 \begin{center}
 \begin{longtable}{p{1.4cm}<{} | p{0.7cm}<{\centering} | p{1.8cm}<{\centering} | p{12cm}<{\centering}}
  \caption{$\calO$-$\mathrm{PQM}$ Jacobians $J/\Q$ with torsion}\\
  $J(\Q)_{\tors}$ & $D$ & $\End(A)_\Q$ & $J=\Jac(C:y^2 = f(x))$ \\
  \cline{1-4}
$(\Z/2\Z)$ & $10$ & $\Q$ & \footnotesize{$ y^2 = -145855x^6 - 729275x^5 + 2187825x^3 - 1312695x $} \\
$(\Z/2\Z)^2$ & $6$ & $\Q(\sqrt{3}) $ & \footnotesize{$ y^2 = -180x^6 - 159x^5 + 894x^4 + 1691x^3 +  246x^2 - 672x + 80 $} \\
$(\Z/3\Z)$ & $15$ & $\Q$ & \footnotesize{$  y^2 =17095x^6 + 345930x^5 + 602160x^4 + 234260x^3 - 43680x^2 - 540930x - 634465 $} \\
$(\Z/3\Z)^2$ & $6$ & $\Q(\sqrt{2})$ & \footnotesize{$ y^2 = -15x^6 - 270x^5 + 315x^4 - 270x^3 - 45x^2 + 270x + 105 $} \\
$(\Z/6\Z)$ & $6$ & $\Q$ & \footnotesize{$ y^2 = 5x^6 + 21x^5 - 63x^4 - 49x^3 + 294x^2 -343 $ }
 \label{PQMequationstable}
 \end{longtable}
 \end{center}

\newcommand{\etalchar}[1]{$^{#1}$}

\end{document}